%% file: main_Arxiv.tex
\documentclass{article}

    \usepackage[preprint,nonatbib]{neurips_2020}

\usepackage[utf8]{inputenc} 
\usepackage[T1]{fontenc}    
\usepackage{hyperref}       
\usepackage{url}            
\usepackage{booktabs}       
\usepackage{amsfonts}       
\usepackage{nicefrac}       
\usepackage{microtype}      

\usepackage{graphicx}
\usepackage{multirow}
\usepackage{amssymb,amsfonts,amsthm}  
\usepackage{latexsym,epsfig}
\usepackage{graphicx} 
\usepackage{dsfont}
\usepackage{lscape}
\usepackage{color}
\usepackage{float}
\usepackage{enumerate}
\usepackage[usenames,dvipsnames]{pstricks}
\usepackage{epsfig}
\usepackage{pst-grad} 
\usepackage{pst-plot} 
\usepackage{mathrsfs}
\usepackage[percent]{overpic}
\usepackage{epstopdf}
\usepackage[misc]{ifsym}
\usepackage{lipsum}
\usepackage{graphicx}
\usepackage{epstopdf}
\usepackage[english]{babel}
\usepackage{array}
\usepackage{anyfontsize}
\usepackage{amsmath}
\usepackage{cite}
\usepackage{algorithm}
\usepackage{algorithmic}
\usepackage{enumitem,kantlipsum}
\usepackage{wrapfig}

\newtheorem{theorem}{Theorem}[section]

\newtheorem{lemma}[theorem]{Lemma}
\newtheorem{proposition}[theorem]{Proposition}
\newtheorem{assumption}[theorem]{Assumption}

\newtheorem{definition}[theorem]{Definition}
\newtheorem{remark}[theorem]{Remark}

\DeclareMathOperator*{\argmin}{arg\,min}

\title{ A Distributed Cubic-Regularized Newton Method for Smooth Convex Optimization over Networks}

\author{%
  C\'esar A. Uribe and Ali Jadbabaie \\
  Laboratory for Information, and Decisions Systems\\
  Institute for Data, Systems, and Society \\
  Massachusetts Institute of Technology\\
  Cambridge, MA 02139 \\
  \texttt{\{cauribe,jadbabai\}@mit.edu} \\
}

\begin{document}
\maketitle

\begin{abstract}
\input{sec_abstract}
\end{abstract} 

\section{Introduction}\label{sec:intro}

\input{sec_introduction}

 
 This paper is organized as follows. Section~\ref{sec:problem} introduces the distributed optimization problem and assumptions and presents the proposed algorithm and its convergence rate analysis. Section~\ref{sec:cubics} describes the distributed approximate solution of the cubic model minimization. Section~\ref{sec:numerics} shows some experimental results. Section~\ref{sec:open} discusses open problems on high-order methods in distributed optimization. Finally, conclusions are presented in Section~\ref{sec:conclusions}.

\input{sec_notation}

\section{Problem Statement, Algorithm, and Main Result}\label{sec:problem}

\input{sec_problem}

\section{Distributed Approximate Minimization of Cubic functions}\label{sec:cubics}

\input{sec_cubic}

\subsection{Primal-Dual Properties for Distributed Implementation over Networks}\label{subsec:dec}

\input{sec_primal_dual}

\subsection{Communication Complexity of the Cubic Approximate Solver}\label{subsec:comm_comp}

\input{sec_communication}

\section{Proof of Theorem~\ref{thm:main_result}: Inexactness in the Estimate Sequence Approach for Cubic Regularization}\label{sec:baes}

\input{sec_main_proof}

\section{Proof of Theorem~\ref{thm:complexity_inner}: Communication Complexity of the Decentralized Cubic Regularized Newton Method}\label{sec:proof_comm}

\input{sec_proof_comm}

\section{Experimental Results}\label{sec:numerics}

\input{sec_experiments_arxiv}

\section{Discussion and Open Problems}\label{sec:open}

\input{sec_discussion}

\section{Conclusions}\label{sec:conclusions}

\input{sec_conclusion}

\section*{Broader Impact}

This work does not present any foreseeable societal consequence.

\begin{ack}
This work was supported by the MIT-IBM AI grant and a Vannevar Bush Fellowship. The authors would like to thank Pavel Dvurechensky, and Alexander Gasnikov for fruitful discussions and comments.
%
%
\end{ack}

\bibliographystyle{abbrv}
\bibliography{all_refs3,example_paper,dist_cubic}

\end{document}

%% file: sec_abstract.tex
We propose a distributed, cubic-regularized Newton method for large-scale convex optimization over networks. The proposed method requires only local computations and communications and is suitable for federated learning applications over arbitrary network topologies. We show a $O(k^{{-}3})$ convergence rate when the cost function is convex with Lipschitz gradient and Hessian, with  $k$ being the number of iterations. We further provide network-dependent bounds for the communication required in each step of the algorithm. We provide numerical experiments that validate our theoretical results.

%% file: sec_introduction.tex
Newton's method for minimizing smooth strongly convex functions has a longstanding history in optimization and scientific computing~\cite{Bertsekas1999,Bennett592}.  The main reason for its popularity is its fast convergence rate. However,  the Newton step's computational cost has often limited its applicability to modern large-scale machine learning problems. Despite these computational challenges, there has been a resurgence of interest in Newton-type algorithms from a theoretical perspective. Over the past two decades, a series of papers by Nesterov and coauthors~\cite{nesterov2006cubic,nesterov2006cubic2} have shown that with appropriate higher-order (e.g., cubic) regularization, such methods achieve provably-fast global convergence rates ~\cite{cartis2011adaptive,nes83,nesterov2013introductory,nemirovski2004interior}.
Additionally, fast higher-order methods have been driven by new insights into their accelerated convergence rates, fundamental limits, and complexity bounds \cite{monteiro2013accelerated,agarwal2018adaptive,pmlr-v99-gasnikov19a,gasnikov2019near}, leading to a series of implementable practical algorithms~\cite{nesterov2018implementable,kamzolov2020near,nesterov2020superfast}. Nevertheless, as mentioned earlier, the impact in modern machine learning applications has been limited~\cite{zhou2019stochastic}. Specially as increasing amounts of data and distributed storage technologies have now driven the need for distributed and federated architectures~\cite{wang2018giant} that split computational cost among many nodes~\cite{hendrikx2020optimal}, e.g., Peer-to-peer federating learning~\cite{lalitha2019peer,pilet2019simple}, distributed optimization methods~\cite{sca17,scaman2018optimal,lan17,Uribe2018a,Rasul19, Mokhtari2016,Zargham2013,ye2020multi,li2018sharp}, MapReduce~\cite{dean2008mapreduce}, Apache Spark~\cite{yang2013trading}, and Parameter Server~\cite{li2014scaling}.

Several second-order distributed methods have been proposed in the literature for smooth, strongly convex functions~\cite{Rasul19,wei2013distributed,jadbabaie2009distributed,Zargham2013, mokhtari2016decentralized}. Nevertheless, such approaches do not provide global convergence rates~\cite{Rasul19} and require strong convexity assumptions to guarantee some linear convergence rate, or require specific master/worker architectures~\cite{Shamir2014,pmlr-v37-zhangb15}.  Other approaches use Quasi-Newton/BFGS-like approaches to compute approximations to the Hessian inverse~\cite{eisen2017decentralized} efficiently, but exact non-asymptotic convergence rates are not available.

The goal of this paper is to address the existing gap in the literature between cubic regularization and distributed optimization. Specifically, motivated buy Empirical Risk Minimization in machine learning applications, we consider the following finite sum minimization problem 
\begin{align}\label{eq:main1}
\min_{x\in Q } \left\lbrace f(x) \triangleq \sum_{i{=}1}^m f^i(x) \right\rbrace , 
\end{align}
 where $f^i$ is the local empirical risk of a subset of data points stored locally by an agent $i$, which means that each agent $i$ has access to the function $f^i(x)$ only. Moreover, we assume the computing units/agents are connected over a network that allows for sparse communication between them. Thus, the proposed solution needs to be executed locally at each agent, using local information only, and achieve the convergence rate as if they had access to the complete dataset. 
 

The key innovation in our solution is to provide a novel analysis for the inexact constrained cubic-regularized Newton method developed in~\cite{baes2009estimate}, to carefully control the errors induced by the disagreement among the nodes in the network, without sacrificing the convergence rate. 

To summarize, the \textit{main contributions} of this paper are as follows: 
\vspace{-0.2cm}
\begin{itemize}[leftmargin=*]
\item We propose a provably-correct (and globally convergent) distributed algorithm based on cubic-regularization. We take into account distributed storage and sparse communications and obtain a convergence rate of $O(k^{{-}3})$. \textit{To the best of the authors' knowledge, this is the first, fully distributed cubic-regularized second-order method that achieves $O(k^{{-}3})$.}
\item We characterize the communication complexity of the proposed algorithm and relate the corresponding approximation error, induced by the sparse communication, to guarantee the desired convergence rate. 
\item We propose a primal-dual distributed method for the minimization of non-separable cubic-regularized second-order functions.
\end{itemize}

%% file: sec_notation.tex
\noindent\textbf{Notation:} Nodes/agents are indexed from $1$ through $m$ (no actual enumeration is needed in the execution of the proposed algorithms). Superscripts $i$ or $j$ denote agent indices and the subscript $k$ denotes the iteration index of an algorithm. $[A]_{ij}$ denotes the entry of the matrix $A$ in its $i$-th row and $j$-th column. $\mathbf{I}_{n}$ denotes the identity matrix of size $n$. For a symmetric non-negative matrix $W$, $\lambda_{\max}(W)$ denotes its largest eigenvalue 
and $\lambda_{\min}^+(W)$ its smallest positive eigenvalue. The condition number of $W$ is denoted as $\chi(W) = \lambda_{\max}{(W)}/ \lambda_{\min}^{+}{(W)}$. The Euclidean norm is denoted as $\|\cdot\|$. $\boldsymbol{1}_n$ is a vector of ones of size $n$, $\otimes$ is the Kronecker product.

%% file: sec_problem.tex
Consider a network of $m$ agents, modeled as a \textit{fixed}, \textit{connected}, and \textit{undirected} graph $\mathcal{G}{=}(V,E)$, where $V {=} (1,\cdots,m)$, and $E \subseteq V\times V$ is a set of edges such that $(j,i)\in E$ if and only if agent $j$ is connected to agent $i$. Agents try to jointly solve~\eqref{eq:main1}, but an agent $i \in V$ has access to $f^i(x)$, $\nabla f^i(x)$, and $\nabla^2 f^i(x)$ only. However, agents are allowed to exchange information over the network $\mathcal{G}$ with its neighbors. We assume each $f^i : Q  \to \mathbb{R}$ is convex with Lipschitz continuous gradient and Hessian, defined in a nonempty, convex, and compact set $Q \subset \mathbb{R}^n$. We further assume without loss of generality that $f$ attains its minimum $f^*$ in the interior of~$Q$.

We can write~\eqref{eq:main1} to introduce the graph $\mathcal{G}$ into the problem formulation~\cite{sca17,lan17,Uribe2018a}. Consider the Laplacian $W_\mathcal{G} \in \mathbb{R}^{m\times m}$ of the graph $\mathcal{G}$, defined as a matrix with entries $[W_\mathcal{G}]_{ij} {=}{-}1$ if $(j,i) \in E$, $[W_\mathcal{G}]_{ij} =\text{deg}(i)$ {if } $i{=} j$, and $[W_\mathcal{G}]_{ij}=0$ otherwise, where $\text{deg}(i)$ is the degree of the node $i$, i.e., the number of neighbors of the node. The matrix $W_\mathcal{G}$ is symmetric and positive semi{-}definite, and $\boldsymbol{1}_m$ is the unique (up to a scaling factor) eigenvector associated with the eigenvalue
$\lambda_{W}^1{=}0$. Thus, for a vector $z\in\mathbb{R}^m$ it holds that $W_\mathcal{G}{z} {=} 0$ if and only if $z_1 {=} \hdots {=} z_m$. If each agent holds a local copy $x^i\in \mathbb{R}^n$ of the decision variable, we obtain the optimization problem:
\begin{align}\label{eq:main2}
\min_{\substack{\mathbf{x} \in Q^m \\   \sqrt{\mathbf{W}}\mathbf{x}   {=} \boldsymbol{0}_{nm} } } \left\lbrace F(\mathbf{x}) \triangleq \sum_{i{=}1}^m f^i(x^i) \right\rbrace , 
\end{align}
where $\mathbf{W} \triangleq W_\mathcal{G} \otimes \mathbf{I}_n$ and $Q^m {=} \{\mathbf{x} \in \mathbb{R}^{nm} \mid \mathbf{x}^\intercal {=} [(x^1) ^\intercal,\cdots, (x^m) ^\intercal]^\intercal, x^i \in Q \ \forall i \in V \} $. 

Problem $\eqref{eq:main2}$ is a reformulation of Problem~$\eqref{eq:main1}$, as the constraint $  \sqrt{\mathbf{W}}\mathbf{x}     {=} \boldsymbol{0}_m$ implies $x^1 {=} \cdots  {=}x^m$. Thus, an optimal point of $\eqref{eq:main2}$ is such that $\mathbf{x}^* {=} \boldsymbol{1}_m \otimes x^* $, where $x^*$ is an optimal point of $\eqref{eq:main1}$. 

Our goal is to  find approximate distributed solutions to Problem $\eqref{eq:main2}$ defined as follows:
\begin{definition}[{\cite[Definition $1$]{lan17}}]\label{def:opt_sol}
    A point $ \hat{\mathbf{x}}$ is called an $(\varepsilon,\tilde{\varepsilon})${-}solution of~\eqref{eq:main2} if $
    F (\hat{\mathbf{x}}) {-} F ^*  
    \leq  \varepsilon$, and $\|\sqrt{\mathbf{W}} \hat{\mathbf{x}}\|_2 \leq \tilde{\varepsilon}$,
    where $F^*$ denotes the optimal value of~\eqref{eq:main2}. 
\end{definition}

Additionally, we define an inexact solution of a constrained optimization problem as:
\begin{definition}\label{def:approx_sol}
    We define a point $\hat x \approx_\delta \argmin_{x \in Q} f(x)$ as a point in $X$ such that $f(\hat x) {-} f^* \leq \delta$, where $f^*$ is the minimum value of the function $f(x)$ over the set $X$. 
\end{definition}

For analysis purposes we define the set $\mathcal{Q}_{\tilde{\varepsilon}} {=} \{ \mathbf{x} \in \mathbb{R}^{nm} \mid \|\sqrt{\mathbf{W}} \hat{\mathbf{x}}\|_2 \leq \tilde{\varepsilon}\}$, which will come handy in later sections. Furthermore, we assume the following conditions are satisfied.


\begin{assumption}[Lipschitz gradient]\label{assum:smooth1}
    Each function $f^i(x)$ is differentiable and has $M^i_1${-}Lipschitz continuous gradients over the set $ Q$, i.e., for any $x,y \in  Q$, $\|\nabla f^i(x) {-} \nabla f^i(y)\| \leq M^i_1\|x {-}y\|$.
\end{assumption}

\begin{assumption}[Lipschitz Hessian]\label{assum:smooth2}
    Each function $f^i(x)$ is twice differentiable and has $M^i_2${-}Lipschitz continuous Hessian over the set $ Q$, i.e., for any $x,y \in  Q$, $
    \|\nabla^2 f^i(x) {-} \nabla^2 f^i(y)\| \leq M^i_2\|x {-}y\|$. Note that $F(\mathbf{x})$ has $M_1${-}Lipschitz gradient and $M_2${-}Lipschitz Hessian, with $M_1 {=} \max_{i\in V} M_1^i$, and $M_2 {=} \max_{i\in V} M_2^i$.
\end{assumption}

\begin{assumption}\label{assum:bounded}
    The diameter of the compact set $Q$ is upper bounded by a constant $D_Q$, i.e., $\max_{x,y \in Q}    \|x {-} y\|  \leq D_Q$.
\end{assumption}

Next, we state our main result. In what follows we show that the distributed Algorithm in ~\ref{alg:main} guarantees that  agents jointly construct a $(\varepsilon,\tilde{\varepsilon})${-}solution to~\eqref{eq:main2} with a convergence rate of $O(k^{-3})$. Algorithm~\ref{alg:main} follows the same structure as the constrained cubic regularized Newton method proposed in~\cite{baes2009estimate}. We define the cubic regularized second order approximation of the function $F(\mathbf{x})$ at a point $\mathbf{z}$ as follows:
\begin{align}\label{eq:aux_subproblem}
&\hat{F}(\mathbf{z},\mathbf{x}) \triangleq  F(\mathbf{z}) {+} \langle \nabla F(\mathbf{z}), \mathbf{x}{-}\mathbf{z} \rangle + \frac{1}{2} \langle \nabla^2 F(\mathbf{z})(\mathbf{x} {-}\mathbf{z}),\mathbf{x} {-}\mathbf{z}\rangle {+} \frac{N}{6}\|\mathbf{x} {-}\mathbf{z}\|^3.
\end{align}

\begin{theorem}[Main Result]\label{thm:main_result}
    Let Assumptions~\ref{assum:smooth1}, \ref{assum:smooth2} and \ref{assum:bounded} hold, $\varepsilon >0$ be a desired accuracy, and $\gamma \in (0,1)$. Moreover, set the number of iterations Algorithm~\ref{alg:main} to $
    K \geq   \lceil{ 12 \varepsilon^{-1/3}\left( F(\mathbf{x}_{0}) {-}F^* {+} \frac{M_2}{6}\|\mathbf{x}_{0}{-}\mathbf{x}^*\|^3 \right)^{1/3} } \rceil$, where $\mathbf{x}^*$ maximizes $R= O(\|\boldsymbol{x}_0 - \boldsymbol{x}^*\|)$, and at every $k\geq 1$, set the accuracy of the auxiliary sub-problems (Lines $10$ and $13$) as
        \begin{align*}
        0&\leq \delta^\phi_k \leq \min \left\lbrace  1, \left(\frac{\left(\frac{\alpha_k \gamma}{1{-}\alpha_k}\right) \varepsilon}{1{+}D_QL_0\big(\frac{6\lambda_k^{2}}{ \sigma_0 } \big)^{1/3}}\right)^3\right\rbrace, \ \  
        0\leq \delta^{F}_k \leq \min \left\lbrace  1, \left(\frac{\left((1{-}\gamma)\alpha_k {+}\frac{1}{2}\right)\varepsilon}{1{+}D_QL_0\big(\frac{3}{\sigma_0} \big)^{1/3}}\right)^3\right\rbrace.
        \end{align*}
  Then, the output of Algorithm~\ref{alg:main}, i.e.,  $\mathbf{x}_K$, is an $(\varepsilon,\varepsilon/R)$- approximate solution of Problem~\eqref{eq:main2}.
\end{theorem}

\begin{algorithm}[ht!]
    \caption{Dec. Cubic Regularized Method}
    \label{alg:main}
    \begin{algorithmic}[1]
        \STATE {\bfseries Input:} $x_0^i{=} \boldsymbol{0}_n$ $\upsilon_0^i{=} x_0^i$, $\lambda_0{=} 1$, $\forall i\in V$.
        \STATE \hspace{1cm} $\phi_0(\mathbf{x}) {=} F(\mathbf{x}_0) {+} M_2\|\mathbf{x}{-}\mathbf{x}_0\|^3/6$.
        \STATE \hspace{1cm} Number of iterations $K$.
        \STATE \textit{Each agent executes the following:}
        \FOR{$k{=}1,\cdots K-1$ }        
        \STATE Find $\alpha_k$ such that $12\alpha_k^3 {=} (1{-}\alpha_k)\lambda_k$.
        \STATE $\lambda_{k{+}1} {=} (1{-}\alpha_k)\lambda_k$.
        \STATE $z^i_{k} {=} \alpha_k \upsilon_k^i {+} (1{-}\alpha_k)x_k^i$.
        \STATE \textit{Use Algorithm~\ref{alg:approx_subproblem} to jointly solve:}
        \STATE $\mathbf{x}_{k{+}1}     \approx_{\delta^F_k} \argmin\limits_{\mathbf{x} \in Q^m \bigcap  \mathcal{Q}_{\tilde\varepsilon}} \hat{F}(\mathbf{z}_k,\mathbf{x})$.
        \STATE $\phi_{k{+}1}(\mathbf{x}) {=} (1{-}\alpha_k)\phi_k(\mathbf{x}) + \alpha_k \big( F(\mathbf{x}_{k{+}1}) + \langle \nabla F(\mathbf{x}_{k{+}1}),\mathbf{x}{-} \mathbf{x}_{k{+}1} \rangle \big)$.
        \STATE \textit{Use Algorithm~\ref{alg:approx_subproblem} to jointly solve:}
        \STATE $\boldsymbol{\upsilon}_{k{+}1}     \approx_{\delta^\phi_k} \argmin\limits_{\mathbf{x} \in Q^m \bigcap  \mathcal{Q}_{\tilde\varepsilon}} \phi_{k{+}1}(\mathbf{x})$.
        \ENDFOR
        \STATE {\bfseries Output:} End points $x_K^i, \ \forall i\in V$.
    \end{algorithmic}
\end{algorithm}

Theorem~\ref{thm:main_result} states that with the appropriate selection of inexactness of the subproblems in Algorithm~\ref{alg:main}, Lines $10$ and $13$, it is possible to obtain a fast convergence rate of~$O(k^{-3})$ in a fully distributed manner. Section~\ref{sec:cubics} shows that such an approximate solution can be computed in a distributed manner via Algorithm~\ref{alg:approx_subproblem}.

\textit{Proof Sketch (Theorem~\ref{thm:main_result})} In~\cite{baes2009estimate}, the authors provide an inexact cubic regularized Newton method with an oracle complexity of $O(\varepsilon^{-1/3})$ for constrained convex problems with Lipschitz Hessian. 
We exploit the dual representation of the cubic terms to build a separable problem amenable to distributed computation. Thus, we bound the communication complexity of the algorithm by primal-dual analysis of the subproblems (Algorithm~\ref{alg:main}, Lines $10$, and $13$). Technically, we show that with appropriate selection of $\delta_k^F$ and $\delta_k^\phi$, if for some $k\geq1$ it holds that $\min_{\mathbf{x} \in Q^m \bigcap  \mathcal{Q}_{\tilde\varepsilon}} \phi_k(\mathbf{x})  \geq F(\hat{\mathbf{x}}_k) - \varepsilon$, then this also holds for $k+1$. Once the inexactness bounds are computed, the oracle complexity of Algorithm~\ref{alg:main} follows from the analysis of the estimating sequences for the particular problem.$\blacksquare$

\begin{algorithm}[ht!]
    \caption{Dec. Approximate Cubic Solver}
    \label{alg:approx_subproblem}
    \begin{algorithmic}[1]
        \STATE {\bfseries Input:} $\mathbf{w}_0^i{=} \boldsymbol{0}_n$, $ \forall i\in V$ $\tilde{\mathbf{w}}_0^i{=} \mathbf{w}_0^i$, $z_k^i$, $\delta >0$.
        \STATE \hspace{1cm} Number of iterations $T$.
        \STATE \textit{Each agent $i$ executes the following:}
        \STATE Compute $\mathbf{g}_i {=} \nabla f^i(z^i_k)$, and $\mathbf{H}_i {=}\nabla^2 f^i(z^i_k)$.
        \STATE  Set $\hat \mu {=} \delta/(2R^2)$, $q {=} \frac{\hat \mu}{M_1 {+} \hat \mu}\frac{\lambda_{\text{min}}^{+}(W)}{\lambda_{\max}(W)}$.
        \STATE  Set $\beta_0$ as the solution to ${\beta_0^2{-}q}{}{=}1{-}\beta_0$.
        \STATE Decompose $\mathbf{H}_i {=} U_i^\intercal\Lambda_i U_i$.
        \FOR{$t{=}0,1,\cdots,T-1$}
        \STATE $\gamma^i {=} U_i\left( \tilde{\mathbf{w}}^i_t {-} \mathbf{g}_i\right)$.
        \STATE Solve $\tau_i^*$ for $ \frac{m}{4} \sum_{j{=}1}^{d} \frac{[\gamma^i]_j^2}{(s_j {+} N\tau_i^* {+} \hat\mu)^2}{=}(\tau_i^*)^2$.
        \STATE $\mathbf{h}_i^*(\tilde{\mathbf{w}}^i_t)  {=}  U_i^\intercal\left( \Lambda_i {+} N\tau_i^*\mathbf{I}_n {+} \hat{\mu}\mathbf{I}_n\right)^{{-}1}\gamma^i$.
        \STATE Share $\mathbf{h}_i^*(\tilde{\mathbf{w}}^i_t)$ with $j$ s.t. $(i,j)\in E$.
        \STATE Receive $\mathbf{h}_j^*(\tilde{\mathbf{w}}^j_t)$ from $j$ s.t. $(j,i)\in E$.
        \STATE $\mathbf{w}_{t{+}1}^i {=} \tilde{\mathbf{w}}_t^i {-} \frac{\hat{\mu}}{\lambda_{\text{max}}(W)} \sum_{j{=}1}^{m}[W]_{ij}\mathbf{h}_j^*(\tilde{\mathbf{w}}^j_t)$.
        \STATE {  $\beta_{t{+}1}^2{{=}} (1{-} \beta_{t{+}1})\beta_{t}^2{+}q\beta_{t{+}1}$, $\beta_{t{+}1}\in (0,1)$.}
        \STATE $\tilde \beta_{t} {=} {\beta_t(1{-}\beta_t)}/({\beta_t^2{+}\beta_{t{+}1}})$.
        \STATE $\tilde{\mathbf{w}}^i_{t{+}1} {=} {\mathbf{w}}^i_t {+} \tilde \beta_t({\mathbf{w}}^i_{t{+}1}{-}{\mathbf{w}}^i_t)$.
        \ENDFOR    
        \STATE {\bfseries Output:} End points $\mathbf{h}_i^*(\tilde{\mathbf{w}}^i_T)+z^i_k, \  \forall i\in V$.
    \end{algorithmic}
\end{algorithm}



We recognize that the  result on $O(\varepsilon^{-1/3})$ oracle calls obtained in Theorem~\ref{thm:main_result} is not optimal.  Second-order methods have been shown to have a lower complexity bound of $O(\varepsilon^{-2/7})$~\cite{agarwal2018adaptive,monteiro2013accelerated}. However, as pointed out in~\cite[Section 4.3.3]{Nesterov2018}, the gain by achieving the optimal rate is bounded by a factor of $O(\varepsilon^{-1/21})$. Therefore, for values of $\varepsilon$ used in practical applications, e.g., $10^{-12}$, the gain is an absolute constant less than $4$. Nevertheless, from a conceptual point of view, getting near-optimal rates remains a valuable open problem. The main difficulty lies in the implementation of distributed line-search procedures, which is an open question in distributed optimization and out of the scope of this paper.

In the next section, we describe the technical details of the proposed approach for the approximate distributed minimization of a cubic regularized second-order model~\eqref{eq:aux_subproblem}.

%% file: sec_cubic.tex
In this section, we study how Algorithm~\ref{alg:approx_subproblem} approximately solves (c.f.~Definition~\ref{def:approx_sol}) a \textit{cubic regularized} second{-}order approximation~\eqref{eq:aux_subproblem} in a distributed manner over a network. We focus on optimization problems of the form 
\begin{align}\label{eq:generic_subproblem}
\min_{\substack{\mathbf{h} \in\mathcal{H} \subset \mathbb{R}^{nm} \\   A\mathbf{h}   {=} \boldsymbol{0}_{nm} }} \left\lbrace \boldsymbol{\Phi}(\mathbf{h}) {\triangleq} \langle \mathbf{g}, \mathbf{h}\rangle {{+}} \frac{1}{2} \langle \mathbf{H}\mathbf{h},\mathbf{h}\rangle  {{+}} \frac{N}{6}\|\mathbf{h}\|^3\right\rbrace,
\end{align}
where $A$ is a generic matrix whose null space is the consensus subspace, i.e., $Ax{=}0 \iff x_i {=} x_j$. Note the subproblems in Lines $10$ and $13$ of Algorithm~\ref{alg:main}, have the form~\eqref{eq:generic_subproblem}. 
Later in this section, we will see the specific details when for a set of points $(z^1_k,\cdots,z^m_k)$ where each $z^i_k$ is stored locally by an agent $i$, we have $\mathbf{g}^\intercal {=} [\mathbf{g}^\intercal_1,\cdots,\mathbf{g}^\intercal_m]$, where $\mathbf{g}_i {=} \nabla f^i(z^i_k)$ for $i \in V$, and $\mathbf{H}{=}{blkdiag}(\mathbf{H}_1,\cdots,\mathbf{H}_m)$\footnote{The function ${blkdiag}(A,\cdots,B)$ generates a block diagonal matrix whose elements are each of the input arguments.} where $\mathbf{H}_i {=} \nabla^2 f^i(z^i_k)$ for $i \in V$, $\mathcal{H} {=}  Q^m$,  $A {=} \sqrt{\mathbf{W}}$ and $\mathbf{h} {=} \mathbf{x}{-}\mathbf{z}$.

Finding a distributed solution to separable problems with linear constraints has been extensively studied in recent literature, due to the flexibility of such an approach in incorporating limited storage and sparse computations~\cite{gas16b,lan17,sca17,Uribe2018a}. However, the main requirement is for the cost function to be separable, i.e., write it as a finite sum of functions. This is true for the first two terms in~\eqref{eq:generic_subproblem} by construction (i.e., the linear and quadratic terms), where one can write $\langle \mathbf{g}, \mathbf{h}\rangle {+} \frac{1}{2} \langle \mathbf{H}\mathbf{h},\mathbf{h}\rangle{=} \sum_{i{=}1}^m \mathbf{g}_i^\intercal \mathbf{h}_i {+} \frac{1}{2}\sum_{i{=}1}^m\mathbf{h}_i^\intercal \mathbf{H}_i \mathbf{h}_i $. Unfortunately, this is not the case for the cubic term $\|\mathbf{h}\|^3$.

Our first task is to exploit the dual structure of the cubic term in~\eqref{eq:generic_subproblem} to construct a surrogate cost function amenable to distributed optimization algorithms. We provide this dual structure for the cubic term in the next proposition.

\begin{proposition}\label{prop:cubic_dual}
    Given some $\mathbf{x}^\intercal {=} [\mathbf{x}^\intercal_1,\cdots,\mathbf{x}^\intercal_m]$ where $\mathbf{x}_i \in \mathbb{R}^n$ for all $i \in V$. Then,
    \begin{align*}
    \frac{1}{3}\|\mathbf{x}\|^3 & {=}  \max_{\substack{  {\tau}_i \geq 0, \  \tau_i   {=} \tau_j \\ i,j \in V}} \left\lbrace \sum_{i{=}1}^{m} \|\mathbf{x}_i\|^2 \tau_i {-}\frac{4}{3m} \sum_{i{=}1}^{m} \tau_i^3 \right\rbrace .
    \end{align*}
\end{proposition}
\begin{proof}
    Projecting on the consensus subspace where $\tau_i   {=} \tau_j$, we have $
    \sum_{i{=}1}^{m} \|\mathbf{x}_i\|^2 \bar\tau {-}(4/3)\bar\tau^3$,
    and by first order optimality conditions $\sum_{i{=}1}^{m} \|\mathbf{x}_i\|^2 {-} 4\bar\tau^2{=}0$. Solving for $\bar\tau$ completes the proof.
\end{proof}

With Proposition~\ref{prop:cubic_dual} at hand, we can rewrite~\eqref{eq:generic_subproblem} as
\begin{align}\label{eq:generic_subproblem_dual}
\min_{\substack{\mathbf{h} \in\mathcal{H} \subset \mathbb{R}^{nm} \\   A\mathbf{h}   {=} \boldsymbol{0}_{nm} }} \max_{\substack{  {\tau}_i \geq 0 \  i \in (1,\cdots,m) \\ B\tau {=} \boldsymbol{0}_m }} \left\lbrace  \langle \mathbf{g}, \mathbf{h}\rangle {+} \frac{1}{2} \langle \mathbf{H}\mathbf{h},\mathbf{h}\rangle  {+} \frac{N}{2}\sum_{i{=}1}^{m} \|\mathbf{h}_i\|^2 \tau_i {-}\frac{2N}{3m} \sum_{i{=}1}^{m} \tau_i^3 \right\rbrace,
\end{align}
where, we have written the consensus constraints on $\tau_i$ by introducing a vector $\tau {=} [\tau_1,\cdots, \tau_m]$ and a generic matrix $B\in \mathbb{R}^{m\times m}$ with $B\tau {=}0 \iff \tau_i {=} \tau_j$ for $i\in(1,\cdots,m)$.

First, in the next lemma we show that the constraint $B\tau {=} \boldsymbol{0}_m$ is not required in \eqref{eq:generic_subproblem_dual}, as the structure of the problem will guarantee a feasible optimal point in $A\mathbf{h}   {=} \boldsymbol{0}_{nm}$ will also be in $B\tau {=} \boldsymbol{0}_m$. This will simplify the analysis for the design of the distributed approximate solver of~\eqref{eq:generic_subproblem}.

\begin{lemma}\label{lemma:no_need_tau}
    An optimal solution pair $(\mathbf{h}^*,\tau^*)$ of 
    \begin{align}\label{eq:generic_subproblem_dual2}
    \min_{\substack{\mathbf{h} \in\mathcal{H} \subset \mathbb{R}^{nm} \\   A\mathbf{h}   {=} \boldsymbol{0}_{nm} }} \max_{\substack{  {\tau}_i \geq 0 \\ i \in (1,\cdots,m)  }} \left\lbrace  \langle \mathbf{g}, \mathbf{h}\rangle {+} \frac{1}{2} \langle \mathbf{H}\mathbf{h},\mathbf{h}\rangle {+} \frac{N}{2}\sum_{i{=}1}^{m} \|\mathbf{h}_i\|^2 \tau_i {-}\frac{2N}{3m} \sum_{i{=}1}^{m} \tau_i^3 \right\rbrace,
    \end{align}
    is also an optimal pair for~\eqref{eq:generic_subproblem_dual}.
\end{lemma}

\begin{proof}
    We can build the Lagrangian function of~\eqref{eq:generic_subproblem_dual} as
    \begin{align*}
    \max_{\mathbf{y}}\min_{ \mathbf{h} \in\mathcal{H}  } \min_{\boldsymbol{\eta}} \max_{\substack{  {\tau}_i \geq 0 \\ i \in (1,\cdots,m)  }} \left\lbrace  \langle \mathbf{g}, \mathbf{h}\rangle {+} \frac{1}{2} \langle \mathbf{H}\mathbf{h},\mathbf{h}\rangle  {+} \frac{N}{2}\sum_{i{=}1}^{m} \|\mathbf{h}_i\|^2 \tau_i {-}\frac{2N}{3m} \sum_{i{=}1}^{m} \tau_i^3 {-}  \langle \mathbf{y}, A\mathbf{h}\rangle   {+} \langle \boldsymbol{\eta}, B\boldsymbol{\tau}\rangle   \right\rbrace.
    \end{align*}
    Thus, the first order optimality conditions are
    \begin{subequations}\label{eq:52}
        \begin{align}
        \label{opt_cond21} A\mathbf{h} & {=} 0,  \\
        \label{opt_cond22} B\boldsymbol{\tau} & {=} 0 , \\
        \label{opt_cond23} \mathbf{g} {+} \left( \mathbf{H} {+} M\mathbf{T}\right) \mathbf{h} {-} A^\intercal\mathbf{y}  &  {=} 0 ,\\
        \label{opt_cond24} \frac{n}{4}\|\mathbf{h}_i\|^2 {-} \tau_i^2 {+} A^\intercal\boldsymbol{\eta} & {=} 0  .
        \end{align}
    \end{subequations}
    Initially, note that~\eqref{opt_cond21} and ~\eqref{opt_cond22} guarantee that all entries of both $\mathbf{h}$ and $\boldsymbol{\tau}$ are equal respectively. This fact, along side ~\eqref{opt_cond24} implies that all the entries of $A^\intercal\boldsymbol{\eta}$ are equal as well. Thus, it follows that $
    A^\intercal\boldsymbol{\eta} {=} \alpha \boldsymbol{1}
    $
    for some value of $\alpha$. It is enough to show that this is true if and only if $\alpha{=}0$.
    
    If $\alpha {=} 0$, then it implies that all entries of $\boldsymbol{\eta}$ are equal to each other, and the solution to both problems are equivalent.  Now assume $\alpha \neq 0$. Initially, we can write the matrix $A$ as $A {=} V \Lambda V^\intercal$ as its eigenvalue decomposition. Thus,
    \begin{align*}
    V\Lambda V^\intercal \boldsymbol{\eta} {=} \alpha \boldsymbol{1} \quad \text{and}, \quad
    \Lambda V^\intercal \boldsymbol{\eta} {=} \alpha V \boldsymbol{1}.
    \end{align*}
    Given that the vector $\boldsymbol{1}$ is the corresponding eigenvector for the eigenvalue $\Lambda_{1,1} {=} 0$, it follows that $
    \Lambda V^\intercal \boldsymbol{\eta} {=} \alpha \boldsymbol{e}_1
    $
    where $\boldsymbol{e}_1$ is the zeroes vector with entry $1$ in its position \mbox{$i{=}1$}. Which implies that $\Lambda_{1,1}[V^\intercal \boldsymbol{\eta}]_1 {=} \alpha$, and since $\Lambda_{1,1} {=} 0$, the only solution is $\alpha {=} 0$, which is a contradiction.
\end{proof}

Now, we are ready to focus on the design of a distributed algorithm for Problem~\eqref{eq:generic_subproblem_dual2}. First, we can define the Lagrangian dual function, for the consensus constraints in $\mathbf{h}$ as
\begin{align*}
\varphi(\mathbf{y}) &{=} \min_{\mathbf{h} \in\mathcal{H} } \max_{\substack{  {\tau}_i \geq 0 \\ i \in (1,\cdots,m)  }} \left\lbrace  \langle \mathbf{g}, \mathbf{h}\rangle {+} \frac{1}{2} \langle (\mathbf{H}{+} N \mathbf{T} \mathbf{h},\mathbf{h}\rangle  {-}\frac{2N}{3m} \sum_{i{=}1}^{m} \tau_i^3  {-} \langle \mathbf{y}, A \mathbf{h} \rangle \right\rbrace,
\end{align*}
where $\mathbf{T} {=} blkdiag(\tau_1\mathbf{I}_n,\cdots, \tau_m\mathbf{I}_n)$, and the dual problem is defined as $\max_{\mathbf{y}} \varphi(\mathbf{y})$.

The dual problem has a number of important properties whose structure we can exploit. For example, since the function $F$ has $M_1${-}Lipschitz gradient, then it follows that the dual function $\varphi(\mathbf{y})$ is $\mu_{\varphi}${-}strongly convex on $\ker(A^T)^{\perp}$ where $\mu_\varphi {=} (\lambda^{+}_{\text{min}}(A^\intercal A)/M_1)$~\cite[ Lemma~$3.1$]{Beck2014}, \cite[Proposition $12.60$]{rockafellar2011variational}, \cite[Theorem~$1$]{nes05}, \cite[Theorem~$6$]{Kakade2009a}.
Moreover, it follows from Demyanov{-}Danskin's theorem~\cite[Proposition $4.5.1$]{Bertsekas2003},  
that \mbox{$\nabla \varphi(\mathbf{y}) {=} A\mathbf{h}^*(A^T\mathbf{y}) $} where $\mathbf{h}^*(A^Ty)$ denotes the unique solution 
of the inner maximization problem 
\begin{align}\label{eq:dual_auxiliary}
\mathbf{h}^*(A^\intercal\mathbf{y})  {=} \argmin_{\mathbf{h}\in\mathcal{H}}  \max_{\substack{  {\tau}_i \geq 0 \\ i \in (1,\cdots,m)  }}  \left\lbrace  \langle \mathbf{g}, \mathbf{h}\rangle {+} \frac{1}{2} \big\langle \left( \mathbf{H} {+} N\mathbf{T}\right) \mathbf{h},\mathbf{h} \big\rangle {-} \frac{2N}{3m} \sum_{i{=}1}^{n}\tau_i^{3}  {-} \langle \mathbf{y}, A\mathbf{h}\rangle  \right\rbrace. 
\end{align}

%% file: sec_primal_dual.tex
At this point, we observe some properties that make the reformulation~\eqref{eq:generic_subproblem_dual2} amenable for a distributed implementation over a network.

The solution of~\eqref{eq:dual_auxiliary} can be computed using local information only at each node, i.e., $\mathbf{h}^*(A^T\mathbf{y})^\intercal {=} [\mathbf{h}^*_1([A^T\mathbf{y}]_1)^\intercal,\cdots,\mathbf{h}^*_m([A^T\mathbf{y}]_m)^\intercal]$ where
    \begin{align}\label{eq:prop1}
    &\mathbf{h}^*_i([A^\intercal\mathbf{y}]_i)  {=} \argmin_{\bar{\mathbf{h}}\in \bar{\mathcal{H}}}  \max_{  {\tau}_i \geq 0}  \left\lbrace  \langle \mathbf{g}_i, \bar{\mathbf{h}}\rangle {-} \frac{2N}{3m} \tau_i^{3} {+} \frac{1}{2} \big\langle \left( \mathbf{H}_i {+} N\tau_i\mathbf{I}_n\right) \bar{\mathbf{h}},\bar{\mathbf{h}} \big\rangle  {-} \langle \mathbf{\mathbf{y}_i}, [A\bar{\mathbf{h}}]_i\rangle  \right\rbrace,
    \end{align}
with the set $\bar{\mathcal{H}}$ is the corresponding marginal set for a single agent only.

 The gradient $A\mathbf{h}^*(A^T\mathbf{y})$ can be computed distributively if the matrix $A$ has the same sparsity pattern as the network. Suppose that $[A]_{ij} \neq 0$ if and only if $(j,i)\in E$. Then, each entry $ [A\mathbf{h}^*(A^T\mathbf{y})]_i$ to be used by an agent $i \in V$, corresponds to a weighted sum of the $ \mathbf{h}^*_j([A^T\mathbf{y}]_j)$ for all other nodes $j\in V$ such that $(j,i)\in E$. That is, the information an agent requires to take gradient steps is available to him via network communications. \textit{The dual function gradient computation corresponds to a communication round over the network.}
 
 Recall that a function is called dual{-}friendly~\cite[Definition $2$]{Uribe2018a}, if we can ``efficiently'' compute (in a closed form or by polynomial time algorithms) a solution to~\eqref{eq:prop1}. In this subsection, we show that our cubic regularized second{-}order approximation~\eqref{eq:aux_subproblem} is indeed dual{-}friendly.

Initially, let us write the optimality conditions of $\eqref{eq:prop1}$. The optimal point $\mathbf{h}^*_i([A^T\mathbf{y}]_i)$ is a solution to the following systems of nonlinear equations $ \mathbf{g}_i {+} \left( \mathbf{H}_i {+} N\tau_i\mathbf{I}_n \right)\mathbf{h}_i {-} [A\bar{\mathbf{h}}]_i $ and $
    {m}/{4}\|\mathbf{h}_i\|^2 {-} \tau_i^2  {=} 0 $. It follows that $\mathbf{h}_i  {=} \left( \mathbf{H}_i {+} N\tau_i\mathbf{I}_n \right)^{{-}1}\left( [A\bar{\mathbf{h}}]_i {-} \mathbf{g}_i\right)$. Moreover, suppose that the matrix $\mathbf{H}_i $ has an eigendecomposition $\mathbf{H}_i {=} U_i^\intercal\Lambda_i U_i$,
where $\Lambda_i$ is a diagonal matrix of eigenvalues $s_1\leq \hdots \leq s_d$ and $U_i$ is an orthonormal matrix of associated eigenvectors. Then $
\mathbf{h}_i  {=}  U_i^\intercal\left( \Lambda_i {+} N\tau_i\mathbf{I}_n \right)^{{-}1}U_i\left( [A\bar{\mathbf{h}}]_i {-} \mathbf{g}_i\right)$. Furthermore, we have $
\|\mathbf{h}_i\|^2  {=} \| U_i^\intercal\left( \Lambda_i {+} N\tau_i\mathbf{I}_n \right)^{{-}1}U_i\left( [A\bar{\mathbf{h}}]_i {-} \mathbf{g}_i\right)\|^2$ and $\| \mathbf{h}_i\|^2 {=} \sum_{j{=}1}^{d} {\gamma_j^2}/{(s_j {+} N\tau_i)^2}$, where $\gamma_j=[U_i\left( [A\bar{\mathbf{h}}]_i {-} \mathbf{g}_i\right)]_j$. Therefore, each agent needs to solve the following nonlinear equation: ${\sqrt{m}}/{2}\|\mathbf{h}_i\| {-} \tau_i {=} 0$. 
However, \cite{conn2000trust} suggest that a simpler approach is to solve the secular equation $
({2}/{\sqrt{m}}){1}/{\|\mathbf{h}_i\|} {-} {1}/{\tau_i} {=} 0
$. 
A comprehensive account of how to efficiently solve the above equation can be found in~\cite[Chapter $7$, Algorithm $7.3.1$]{conn2000trust}, or in~\cite{carmon2018analysis}. Thus, we assume that each agent can locally and efficiently find a solution.

The dual function $\varphi(\mathbf{y})$ is strongly convex on a defined subspace, but it is non{-}smooth. One can use traditional approaches for non{-}smooth minimization\cite{goffin1977convergence,shor2012minimization}. However, we make the design choice of exploiting the max structure of the function by using Nesterov's dual smoothing approach~\cite{nes05} which has been shown optimal for the problem class of non{-}smooth minimization, specially for dual-friendly problems. To do so, we define a regularized problem 
\begin{align}\label{eq:dual_regu}
\hat\varphi(\mathbf{y}) &{=} \min_{\mathbf{h} \in\mathcal{H} } \max_{\substack{  {\tau}_i \geq 0 \\ i \in (1,\cdots,m)  }} \left\lbrace  \langle \mathbf{g}, \mathbf{h}\rangle {+} \frac{1}{2} \langle (\mathbf{H}{+} N \mathbf{T} \mathbf{h},\mathbf{h}\rangle   {-}\frac{2N}{3m} \sum_{i{=}1}^{m} \tau_i^3  {-} \langle \mathbf{y}, A \mathbf{h} \rangle  {+} \frac{\hat{\mu}}{2}\|\mathbf{h}\|^2\right\rbrace,
\end{align}
where we have added a quadratic term to our cost function to induce smoothness in the dual space. Moreover, an appropriate selection of $\hat{\mu}$ can provide bounds that relate to the original non{-}regularized function, see~\cite[ Proposition $5.2$]{Uribe2018a}, and~\cite[ Lemma $3$]{gas16b}. Particularly, if $\hat\mu \leq \delta/(2R^2)$, where $R {=} \|\mathbf{h}^*(A^\intercal\mathbf{y}^*)\|$, and $\mathbf{y}^*$ denotes the smallest norm solution of the non regularized problem, and $\delta >0$ is the desired accuracy. Then, an approximate solution point $\hat{\mathbf{y}}$ such that $\hat\varphi(\hat{\mathbf{y}}) {-} \hat\varphi^* \leq \delta/2$ implies  $\varphi(\hat{\mathbf{y}}) {-} \varphi^* \leq \delta$, where $\hat\varphi^*$ and $\varphi^*$ are the optimal values of the regularized and non{-}regularized functions respectively. Therefore, the smoothed dual function $\hat \varphi (\mathbf{y})$ is $\mu_{\hat{\varphi}}${-}strongly concave and has $M_{\hat{\varphi}}${-}Lipschitz continuous gradients, where $\mu_{\hat{\varphi}}{=}{\lambda_{\min}^{{+}}(A^\intercal A)}/{(L_0 {+} \hat\mu)}$ and $M_{\hat{\varphi}}{=}{\lambda_{\max}(A^\intercal A)}/{\hat{\mu}}$. Having a strongly convex function with Lipschitz gradient allows for the use of traditional Fast Gradient Methods~\cite{nes13}. More importantly, this regularization approach do not affect the decentralization properties.

%% file: sec_communication.tex
In this subsection, build upon recently develop dual{-}based optimal algorithms for dual{-}friendly functions~\cite{Uribe2018a} to provide an approximate solution to the auxiliary Subproblem~\eqref{eq:dual_regu}. 

\begin{theorem}\label{thm:complexity_inner}
    Let Assumptions~ \ref{assum:smooth1}, and \ref{assum:smooth2} hold. For any $\delta >0$, set the number of iterations in~Algorithm~\ref{alg:approx_subproblem} as
        \begin{align*}
        T {\geq} 2\sqrt{\left(\frac{2M_1 R^2_\varphi}{\delta}{{+}}1\right)\chi(W)}\log\left(\frac{8\sqrt{2}\lambda_{{\max}}(W)R^2_\varphi R_{\mathbf{h}}^2}{\delta^2}\right),
        \end{align*}
    where  $\chi(W) {=} \lambda_{\max}{(W)}/ \lambda_{\min}^{{{+}}}{(W)}$, $R_{\mathbf{h}} = O( \|\mathbf{h}^* - \mathbf{h}^*
    (0)\|)$, $R_\varphi = O( \|\mathbf{w}^*\|)$ are bounds on the distance to the optimal solution and the initial point for the primal and dual variables. Then,  the output $ \mathbf{h}^*(\tilde{\mathbf{w}}_T)$ of Algorithm~\ref{alg:approx_subproblem} is an $(2\delta,\delta/R_\varphi)$-approximate solution of~\eqref{eq:generic_subproblem}.
\end{theorem}

The result in Theorem~\ref{thm:complexity_inner} shows that $\tilde O(\sqrt{\chi(W)/\delta})$ communication rounds on the network are needed to reach an $(\delta,\delta/R)$ approximate solution of~\eqref{eq:aux_subproblem}.   Moreover, this can be done in a fully distributed manner. 

\begin{remark}
    Note that we have followed one particular approach in~\cite{Uribe2018a} to solve the smooth inner problem. However, there are other algorithms with similar convergence rate guarantees, for example, ~\cite{lan2018communication,ye2020multi,li2018sharp}. It follows from~\cite[Lemma 1]{nesterov2015universal}, \cite[Lemma 1]{dvurechensky2017gradient}, or \cite[Corollary 18.14]{bauschke2011convex} that uniform convexity of the function $\boldsymbol{\Phi}(\cdot)$ implies H\"older continuity of the dual function $\varphi(\cdot)$, with order $\nu = 1/2$ and parameter $M_\nu = 1/\sqrt{\sigma_0}$. Therefore, one can use more sophisticated methods~\cite{yashtini2016global} to improve the communication complexity for the solution of the sub-problem~\eqref{eq:generic_subproblem}. For example, the recently proposed Universal Intermediate Gradient Method~\cite{Kamzolov2020}.
\end{remark}

%% file: sec_main_proof.tex
Our goal in this section is to prove Theorem~\ref{thm:main_result}, we extend the results of estimate sequences of Baes~\cite{baes2009estimate} to take into account inexactness coming from approximate solutions of the auxiliary subproblems and provide a communication complexity to Algorithm~\ref{alg:main}. To do so, we start with Algorithm~\ref{alg:baes_modified}, which is a modified version of Baes' Cubic Regularized method. The main difference between Algorithm~\ref{alg:baes_modified} and Baes' constrained cubic regularized Newton's method~\cite[Algorithm 4.1]{baes2009estimate} is that we define inexactness in both subproblems according to Definition~\ref{def:approx_sol}. That is, in terms of distance to optimality measured by function value. We are allowed to make such analysis due to the specific structure induced by the problems we are required to solve and the algorithms we have available for computing such an approximate solution.

\begin{algorithm}[ht!]
	\caption{Modified Baes' Cubic Regularization}
	\label{alg:baes_modified}
	\begin{algorithmic}
		\STATE {\bfseries Input:} $\hat x_0{=} 0$, $\hat \nu_0{=} \hat x_0$, $\lambda_0{=} 1$.
		\STATE \hspace{1cm} $\phi_0(x) {=} f(\hat x_0) {+} M_2\|x{-}\hat x_0\|^3/6$.
		\STATE \hspace{1cm} Number of iterations $K$.
		\FOR{$k{=}1,\cdots K$ }        
		\STATE Find $\alpha_k$ such that $12\alpha_k^3 {=} (1{-}\alpha_k)\lambda_k$.
		\STATE Set $\lambda_{k{+}1} {=} (1{-}\alpha_k)\lambda_k$.
		\STATE Set $z_{k} {=} \alpha_k \hat \nu_k {+} (1{-}\alpha_k)\hat x_k$.
		\STATE \textit{Approximately solve:}
		\STATE Set $\hat x_{k{+}1}     \approx_{\delta_k^f} \argmin\limits_{x \in Q} \hat{f}(z_k,x)$.
		\STATE Set $\phi_{k{+}1}(x) {=} (1{-}\alpha_k)\phi_k(x) + \alpha_k \big( f(\hat x_{k{+}1}) +  \langle \nabla f(\hat x_{k{+}1}),x{-} \hat x_{k{+}1} \rangle \big)$.
		\STATE \textit{Approximately solve:}
		\STATE Set $\hat \nu_{k{+}1}     \approx_{\delta_k^\phi} \argmin\limits_{x \in Q} \phi_{k{+}1}(x)$.
		\ENDFOR
	\end{algorithmic}
\end{algorithm}

The idea of estimate sequences was first introduced by~\cite{nes83,nes05} and later extended in~\cite{Auslender2006}. Baes~\cite{baes2009estimate} shed some light on the use of estimate sequences for the design of high{-}order optimization algorithms that generalized first{-}order methods. We follow the estimate sequence approach in \cite{baes2009estimate} to prove the convergence rate properties of Algorithm~\ref{alg:main}.

For simplicity of notation, we will consider the generic problem $\min_{x\in Q } f(x)$ for a compact, convex and bounded set, and a convex function $f$ with $M_1$-Lipschitz gradient and $M_2$-Lipschitz Hessian. To do so, we provide a slightly modified cubic regularized Newton method based on estimated sequences, introduced in Algorithm~\ref{alg:baes_modified}. Later on, we will provide the specific result Problem~\eqref{eq:main2}. 

Note that Algorithm \ref{alg:baes_modified} is different from the cubic regularized Newton method proposed in~\cite{baes2009estimate} because it is stated in terms of function value suboptimality in both the subproblems.

Recall a couple of definitions and properties for estimate sequences.

\begin{definition}[Chapter $2$ in \cite{nes13}]\label{def:estimate}
    An estimate sequence for the function $f(x)$ is a sequence of convex functions $(\phi_k)_{k\geq 0}$ and a sequence of positive numbers $(\lambda_{k})_{k\geq 0}$ satisfying: $\lim_{k\to 0} \lambda_k {{=}}0$ and $\phi_k(x) \leq (1{{-}}\lambda_k)f(x) {{+}} \lambda_k\phi_0(x)$ , 
    for all $x \in Q$ for $k\geq 1$.
\end{definition}

Estimate sequences provide an understanding of the convergence rate of a sequence of iterates generated by some arbitrary algorithm, as described in the next proposition. 

\begin{proposition}[Adapted from Proposition $2.1$ in~\cite{baes2009estimate}]\label{prop:estimate_sequence}
    Suppose that a sequence of iterates $(x_k)_{k\geq 0}$ in $Q$ satisfies $f(x_k) -\varepsilon \leq \min_{x\in Q} \phi_k(x)$, and $\varepsilon >0$. Then, $f(x_k) {-} f^* \leq \lambda_k(\phi_0(x^*) {-}f^*) + \varepsilon$ for $k \geq 1$.
\end{proposition}
\begin{proof}
    It follows from the Definition~\ref{def:estimate} that
    \begin{align*}
    f(x_k){-}\varepsilon&\leq \min_{x\in Q} \phi_k(x) \leq \min_{x\in Q} f(x) {+} \lambda_k(\phi_0(x) {-} f(x)) \\
    &\leq  f^* + \lambda_k(\phi_0(x^*) {-} f(x^*)).
    \end{align*}
\end{proof}

For example, if Assumption~\ref{assum:smooth2} holds, one useful way to construct an estimate sequence is:
\begin{align}\label{eq:estimates}
\phi_0(x) &{=} f(x_0) {+}  {M_2}/{6}\|x{-} x_0\|^3, \quad \text{and} \quad
\phi_k(x)  {=} (1{-}\alpha_k)\phi_k {+} \alpha \left(f(y_k){+} \langle \nabla f(y_k), x{-}y_k \rangle\right)
\end{align} 
for a given starting point $x_0\in Q$ and an appropriate choice of $(\alpha_k)_{k \geq 0}$ and $(y_k)_{k \geq 0}$. Additionally, $\lambda_0 =1$, and $\lambda_{k+1}= \lambda_k(1{-}\alpha_k)$ for a sequence $(\alpha_k)_{k\geq 0}$ whose sum diverges.

Proposition~\ref{prop:estimate_sequence} provides an insight, which as pointed out in~\cite{baes2009estimate}, indicates that one key element in the use of estimate sequences is for an algorithm to be able to construct a sequence $(x_k)_{k \geq 0}$ for which $f(x_k) \leq \min_{x\in Q} \phi_k(x)$ holds or $f(x_k) -\varepsilon \leq \min_{x\in Q} \phi_k(x)$ in the inexact case.

\begin{proposition}\label{prop:primal_dual}
Let $\boldsymbol{\Phi}(\mathbf{h})$ be defined in~\eqref{eq:generic_subproblem}, then it holds that for all $x,y\in \mathcal{H} $
\begin{align}\label{eq:prop_subproblem}
\frac{L_0 +M_1}{2}\|y-x\|^2 \geq \boldsymbol{\Phi}(y) - \boldsymbol{\Phi}(x) - \langle \boldsymbol{\Phi}'(x),y-x \rangle \geq \frac{\sigma_0}{3}\|y-x\|^3,
\end{align}
where $L_0 = M_2D_Q$, and $\sigma_0 = M_2/{6}$.
\end{proposition}

\begin{proof}
    From~\cite[Lemma 8.2]{baes2009estimate} with $p=3$ in our case, we have that
\begin{align*}
\phi(x) = \|x-x_0\|^p
\end{align*}
then
\begin{align*}
\phi(y) - \phi(x) - \langle \phi'(x),y-x \rangle \geq c_p\|x-y\|^p
\end{align*}
where
\begin{align*}
c_p = \frac{p-1}{\big( (2p-3)^{\frac{1}{p-2}}+1 \big)^{p-2}}.
\end{align*}
Moreover, from~\cite[Lemma 5.1]{baes2009estimate} we have
\begin{align*}
\phi_0(x) = f(x_0) + \frac{M}{p!}\|x-x_0\|^p
\end{align*}
then
\begin{align*}
L_0\|y-x\|^2 \geq \langle \phi'_0(y) -\phi'_0(x) , y-x \rangle \geq \sigma_0\| y-x\|^p,
\end{align*}
where
\begin{align*}
L_0 = M\frac{D_Q^{p-2}}{(p-2)!} \qquad \text{and} \qquad \sigma_0 = 2M \frac{c_p}{p!}.
\end{align*}

    Finally, it follows from~\eqref{eq:generic_subproblem}, that $\boldsymbol{\Phi}(x) $ is uniformly strongly convex of order $p=3$. Moreover, since we assume $F(x)$ is $M_1$ smooth, we have that  $\boldsymbol{\Phi}(x) $ is $M_1 + L_0$ smooth.
\end{proof}

Next, we show that Algorithm~\ref{alg:baes_modified} builds an estimate sequence, and furthermore, one can appropriately chose the accuracy of each of the subproblems, such that the error does not accumulate and we can apply Proposition~\ref{prop:estimate_sequence} for the convergence rate analysis. In particular, and considering Problem~\eqref{eq:main2}, following a construction of an estimate sequence as suggested in~\eqref{eq:estimates}, we seek to inductively prove for the output sequence $(\mathbf{x}_k)_{k \geq 0}$ of Algorithm~\ref{alg:main}, if $\min_{\mathbf{x} \in Q^m \bigcap  \mathcal{Q}_{\tilde\varepsilon}}\phi_k(\mathbf{x})\geq F(\mathbf{x}_k) - \varepsilon$, then $\min_{\mathbf{x} \in Q^m \bigcap  \mathcal{Q}_{\tilde\varepsilon}}\phi_{k+1}(\mathbf{x})\geq F(\mathbf{x}_{k+1}) - \varepsilon$.

The next lemma provides bounds for the accuracy of solving each of the subproblems in Algorithm~\ref{alg:baes_modified} such that we can apply Proposition~\ref{prop:estimate_sequence}.

\begin{lemma}\label{lemma:main}
    Let Assumptions~\ref{assum:smooth1}, \ref{assum:smooth2} and \ref{assum:bounded} hold. Let $\varepsilon >0$ and $\gamma \in (0,1)$, and assume that for a fixed $k \geq 0$ and points $\hat{x}_k,\hat{\nu}_k \in Q$: $\min_{x\in Q} \phi_k(x) \geq f(\hat{x}_k) - \varepsilon $,
    with 
    {\small
\begin{align*}
        0&\leq \delta^\phi_k \leq \min \left\lbrace  1, \left(\frac{\left(\frac{\alpha_k \gamma}{1{-}\alpha_k}\right) \varepsilon}{1{+}D_QL_0\big(\frac{6\lambda_k^{2}}{ \sigma_0 } \big)^{1/3}}\right)^3\right\rbrace, \qquad 
        0\leq \delta^{f}_k \leq \min \left\lbrace  1, \left(\frac{\left((1{-}\gamma)\alpha_k {+}\frac{1}{2}\right)\varepsilon}{1{+}D_QL_0\big(\frac{3}{\sigma_0} \big)^{1/3}}\right)^3\right\rbrace .
        \end{align*}
    }
    Then $
    \min_{x\in Q} \phi_{k+1}(x) \geq  f(\hat{x}_{k+1})  - \varepsilon$.
\end{lemma}

\begin{proof}
    Initially, by definition of the sequence $(\phi_k(x))_{k\geq 0}$ in~\eqref{eq:estimates},
    \begin{align}\label{eq:aux1}
    & \min_{x\in Q} \phi_{k+1}(x)  = \min_{x\in Q} \{ (1-\alpha_k)\phi_k(x) + \alpha_k \big( f(\hat{x}_{k+1}) + \langle \nabla f (\hat{x}_{k+1}) , x - \hat{x}_{k+1} \rangle \big) \} \nonumber\\
    & \geq \min_{x\in Q} \{ (1-\alpha_k)\big( \phi_k (\hat{\nu}_k) + \langle \nabla \phi_k(\hat{\nu}_k),x - \hat{\nu}_k \rangle + \lambda_k \frac{M_2}{6}\|x-\hat{\nu}_k\|^3 \big)  + \nonumber\\
    & \qquad + \alpha_k \big( f(\hat{x}_{k+1}) + \langle \nabla f (\hat{x}_{k+1} ), x - \hat{x}_{k+1} \rangle \big)   \} \nonumber \\
    & = \min_{x\in Q} \{ (1-\alpha_k)\big( \phi_k (\hat{\nu}_k) + \langle \nabla \phi_k(\hat{\nu}_k),x - \hat{\nu}_k \rangle \big)+ \lambda_{k+1} \frac{M_2}{6}\|x-\hat{\nu}_k\|^3  + \nonumber \\ 
    & \qquad \alpha_k \big( f(\hat{x}_{k+1}) + \langle \nabla f (\hat{x}_{k+1}) , x - \hat{x}_{k+1} \rangle \big)   \}, 
    \end{align}
    where the inequality in the second line follows from Lemma~$8.2$ in~\cite{baes2009estimate}, that shows that $\phi_k(y) \geq \phi_k(x) + \langle \nabla \phi_k(x),y-x\rangle + \lambda_{k}{M_2}/{6}\|y-x\|^3$ for all $x,y \in Q$ and $k\geq 0$, and the last equality from the definition of $\lambda_{k+1}$.
    
    Now, we focus on bounding the term $\langle \nabla \phi_k(\hat{\nu}_k),x - \hat{\nu}_k \rangle$. Initially, by adding and subtracting $\nabla \phi_k({\nu}_k)$ we have
    \begin{align*}
    & \min_{x\in Q} \langle \nabla \phi_k(\hat{\nu}_k),x - \hat{\nu}_k \rangle = \min_{x\in Q} \langle \nabla \phi_k({\nu}_k),x - \hat{\nu}_k \rangle - \langle \nabla \phi_k({\nu}_k) - \nabla \phi_k(\hat{\nu}_k),x - \hat{\nu}_k \rangle\\
    & \qquad \geq   \min_{x\in Q} \{ \langle \nabla \phi_k({\nu}_k),x - \hat{\nu}_k \rangle  - \|\nabla \phi_k(\hat{\nu}_k) - \nabla \phi_k({\nu}_k) \| \|x - \hat{\nu}_k\|  \}\\
    & \qquad  \geq  \min_{x\in Q} \{\langle \nabla \phi_k({\nu}_k),x - {\nu}_k \rangle  + \langle \nabla \phi_k({\nu}_k),\nu_k - \hat{\nu}_k \rangle  - \|\nabla \phi_k(\hat{\nu}_k) - \nabla \phi_k({\nu}_k) \| \|x - \hat{\nu}_k\|  \},
    \end{align*}
    where the first inequality follows from Cauchy--Schwarz inequality, and the second one by adding and subtracting $\nu_k$. Next, given that the function $\phi_0$ has Lipschitz gradients with constant $M_2D_Q$ (see Lemma $5.1$ in \cite{baes2009estimate}) where $D_Q$ is the diameter of the set $Q$, it holds that
    \begin{align*}
    \min_{x\in Q} \langle \nabla \phi_k(\hat{\nu}_k),x - \hat{\nu}_k \rangle & \geq  \min_{x\in Q} \{\langle \nabla \phi_k({\nu}_k),x - {\nu}_k \rangle  + \langle \nabla \phi_k({\nu}_k),\nu_k - \hat{\nu}_k \rangle  - M_2D_Q\lambda_k \|\hat{\nu}_k - \nu_k\| D_Q   \}, \\
    & \qquad \geq \langle \nabla \phi_k({\nu}_k),\nu_k - \hat{\nu}_k \rangle  - M_2D_Q\lambda_k \|\hat{\nu}_k - \nu_k\| D_Q \\
    & \qquad \geq  \phi_k(\nu_k) - \phi_k(\hat{\nu}_k) + \lambda_k  \frac{M_2}{6}\|\nu_k-\hat{\nu}_k\|^3 - M_2D_Q\lambda_k \|\hat{\nu}_k - \nu_k\| D_Q,
    \end{align*}
    where the second inequality follows from the constrained optimality conditions for the function $\phi_k$, recall that $\nu_k$ is defined as the minimizer of $\phi_k$ on $Q$. Thus, the first-order optimality condition reads as $\langle \nabla \phi_k({\nu}_k),x - \hat{\nu}_k \rangle \geq 0$ for all $x \in Q$. The third inequality follows again from~\cite[Lemma~$8.2$]{baes2009estimate}.
    
Assuming the accuracy of the approximate solution $\hat{\nu}_k$ is such that $\phi_k(\hat{\nu}_k) - \phi_k(\nu_k) \leq \delta^\phi_k$. Then,
    \begin{align*}
    \min_{x\in Q} \langle \nabla \phi_k(\hat{\nu}_k),x - \hat{\nu}_k \rangle & \geq -\delta^\phi_k  - M_2D_Q\lambda_k \|\hat{\nu}_k - \nu_k\| D_Q,
    \end{align*}
    where we have removed the positive term in the upper bound. Moreover, we can express the last term $\|\hat{\nu}_k - \nu_k\|$ in terms of the accuracy $\delta^\phi_k$ since it follows from \cite[Lemma $5.1$]{baes2009estimate} that:
    \begin{align*}
    \delta^\phi_k \geq \phi_k(\hat{\nu}_k) - \phi_k(\nu_k) \geq \frac{\sigma_0\lambda_k}{p}\|\hat{\nu}_k - \nu_k\|^3,
    \end{align*}
    from which we obtain the bound:
    \begin{align*}
    \min_{x\in Q} \langle \nabla \phi_k(\hat{\nu}_k),x - \hat{\nu}_k \rangle & \geq -\delta^\phi_k  -L_0D_Q\lambda_k \left(\frac{p\varepsilon_\phi}{\sigma_0 \lambda_k}\right)^{1/3}.
    \end{align*}
    Furthermore, assuming $\varepsilon_\phi \leq 1$ without loss of generality, we have that
    \begin{align}\label{eq:aux2}
    \min_{x\in Q} \langle \nabla \phi_k(\hat{\nu}_k),x - \hat{\nu}_k \rangle & \geq -(\delta^\phi_k )^{1/3}\left(1+ L_0D_Q\lambda_k \left(\frac{p}{\sigma_0 \lambda_k}\right)^{1/3} \right)\geq -\bar{\delta}^\phi_k ,
    \end{align}
    for an appropriate selection of the error $\delta^\phi_k $.

Lets recall~\eqref{eq:aux1}, and use the bound~\eqref{eq:aux2}, then
\begin{align*}
&\min_{x\in Q} \phi_{k+1}(x)   \geq \min_{x\in Q} \{ (1-\alpha_k)\big( \phi_k (\hat{\nu}_k) -\bar{\delta}^\phi_k \big)+ \lambda_{k+1} \frac{M_2}{6}\|x-\hat{\nu}_k\|^3  + \\
& \qquad + \alpha_k \big( f(\hat{x}_{k+1}) + \langle \nabla f (\hat{x}_{k+1}) , x - \hat{x}_{k+1} \rangle \big)   \} \\
&  \geq (1-\alpha_k)\big( \phi_k (\hat{\nu}_k) -\bar{\delta}^\phi_k \big)+ \min_{x\in Q} \{  \lambda_{k+1} \frac{M_2}{6}\|x-\hat{\nu}_k\|^3  + \alpha_k \big( f(\hat{x}_{k+1}) + \langle \nabla f (\hat{x}_{k+1}) , x - \hat{x}_{k+1} \rangle \big)   \} \\
&\geq (1-\alpha_k)\big( \phi_k (\hat{\nu}_k)- \bar{\delta}^\phi_k \big) + \alpha_k \big( f(\hat{x}_{k+1} )+ \langle \nabla f (\hat{x}_{k+1}) , \hat{\nu}_k - \hat{x}_{k+1}\rangle \big)+ \\
& \qquad + \min_{x\in Q} \{ \lambda_{k+1} \frac{M_2}{6}\|x -\hat{\nu}_k\|^3  + \alpha_k\langle \nabla f (\hat{x}_{k+1}) , x - \hat{\nu}_k\rangle   \},
\end{align*}
where in the last inequality we have added and subtracted $\hat{\nu}_{k}$. Now, from the hypotheses that $\min_{x\in Q}\phi_k(x) \geq f(\hat{x}_k)-\varepsilon$, we have
\begin{align*}
&\min_{x\in Q} \phi_{k+1}(x) \geq (1-\alpha_k)\big( f (\hat{x}_k)-\varepsilon - \bar{\delta}^\phi_k\big) + \alpha_k \big( f(\hat{x}_{k+1} )+ \langle \nabla f (\hat{x}_{k+1}) , \hat{\nu}_k - \hat{x}_{k+1}\rangle \big) \\
& \qquad + \min_{x\in Q} \{ \lambda_{k+1} \frac{M_2}{6}\|x -\hat{\nu}_k\|^3  + \alpha_k\langle \nabla f (\hat{x}_{k+1}) , x - \hat{\nu}_k\rangle   \}, \\
&\geq (1-\alpha_k)\big(  (f(\hat{x}_{k+1}) + \langle \nabla f(\hat{x}_{k+1}) , \hat{x}_k - \hat{x}_{k+1}\rangle ) -\varepsilon - \bar{\delta}^\phi_k \big) + \alpha_k \big( f(\hat{x}_{k+1} )+ \langle \nabla f (\hat{x}_{k+1}) , \hat{\nu}_k - \hat{x}_{k+1}\rangle \big) \\
& \qquad \qquad+ \min_{x\in Q} \{ \lambda_{k+1} \frac{M_2}{6}\|x -\hat{\nu}_k\|^3  + \alpha_k\langle \nabla f (\hat{x}_{k+1}) , x - \hat{\nu}_k\rangle   \}, \\
&\geq f(\hat{x}_{k+1}) +  (1-\alpha_k)\big(  \langle \nabla f(\hat{x}_{k+1}) , \hat{x}_k - \hat{x}_{k+1}\rangle  -\varepsilon - \bar{\delta}^\phi_k \big) + \alpha_k\langle \nabla f (\hat{x}_{k+1}) , \hat{\nu}_k - \hat{x}_{k+1}\rangle \\
& \qquad \qquad+ \min_{x\in Q} \{ \lambda_{k+1} \frac{M_2}{6}\|x -\hat{\nu}_k\|^3  + \alpha_k\langle \nabla f (\hat{x}_{k+1}) , x - \hat{\nu}_k\rangle   \},
\end{align*}
where the second to last inequality follows form a linear lower bound of the function $f$ at the point $f(\hat{x}_k)$, and the last one follows form eliminating common terms. Rearranging some terms, and defining $z_k = (1-\alpha_k)\hat{x}_k  +\alpha_k \hat{\nu}_k$, we obtain:
\begin{align*}
&\min_{x\in Q} \phi_{k+1}(x) \geq f(\hat{x}_{k+1}) +   \langle \nabla f(\hat{x}_{k+1}) , z_k -\hat{x}_{k+1}\rangle  -(1-\alpha_k)\big( \varepsilon +\bar{\delta}^\phi_k \big)  \\
& \qquad \qquad+ \min_{x\in Q} \{ \lambda_{k+1} \frac{M_2}{6}\|x -\hat{\nu}_k\|^3  + \alpha_k\langle \nabla f (\hat{x}_{k+1}) , x - \hat{\nu}_k\rangle   \}.
\end{align*}

Next, by Lemma~$4.3$ in~\cite{baes2009estimate}, we have
\begin{align}\label{eq:aux5}
&\min_{x\in Q} \phi_{k+1}(x) \geq f(\hat{x}_{k+1}) +   \langle \nabla f(\hat{x}_{k+1}) , z_k -\hat{x}_{k+1}\rangle  -(1-\alpha_k)\big( \varepsilon +\bar{\delta}^\phi_k \big)  \nonumber \\
& \qquad \qquad+ \min_{x\in Q} \{ \frac{\lambda_{k+1}}{\alpha^3_k} \frac{M_2}{6}\|x -z_k\|^3  + \langle \nabla f (\hat{x}_{k+1}) , x - z_k\rangle   \},  \nonumber\\
& \geq f(\hat{x}_{k+1})  -(1-\alpha_k)\big( \varepsilon +\bar{\delta}^\phi_k \big) + \min_{x\in Q} \{ \frac{\lambda_{k+1}}{\alpha^3_k} \frac{M_2}{6}\|x -z_k\|^3  + \langle \nabla f (\hat{x}_{k+1}) , x - \hat{x}_{k+1}\rangle   \},
\end{align}
where the last inequality follows by adding and subtracting $\hat{x}_{k+1}$.

The next step is to bound the last term in the above relation, i.e., $\langle \nabla f (\hat{x}_{k+1}) , x - \hat{x}_{k+1}\rangle$, from the fact that $\hat{x}_{k+1}$ is an approximate solution to the auxiliary sub problem
\begin{align*}
\argmin_{x \in Q} \left\lbrace  \varphi(x) \triangleq \langle \nabla f(z), x-z \rangle + \frac{1}{2} \langle \nabla^2 f(z)(x-z),x-z \rangle + \frac{N}{6}\|x-z\|^3 \right\rbrace.
\end{align*}
However, note that the optimality condition  $\langle \nabla \varphi(x_{k+1}), y-x_{k+1} \rangle  \geq 0$ for all $y \in Q$ holds for an optimal point $x_{k+1}$, but we only have access to inexact solvers which return a point $\hat{x}_{k+1} \in Q$ such that $\varphi(\hat{x}_{k+1}) - \varphi({x}_{k+1}) \leq  \delta^{f}_k $. So we will see what is the effect of this error.

We will use Proposition~\ref{prop:primal_dual} to relate error in the function value to error in the gradient.
\newline
Assume we are able to optain a point $\tilde{x}$ such that
\begin{align*}
\varphi(\tilde{x}) - \min_{x\in Q}\varphi(x) \leq \delta^{f}_k .
\end{align*}
$\delta^{f}_k $ is the accuracy by which we solve the subproblem of the cubic regularization method.
\newline
\begin{align}\label{eq:aux_5}
\langle  \varphi'(\tilde{x}),x - \tilde{x} \rangle & = \langle  \varphi'(x^*),x - \tilde{x} \rangle - \langle  \varphi'(x^*) - \varphi'(\tilde{x} ),x - \tilde{x} \rangle \nonumber \\
& \geq \langle  \varphi'(x^*),x - \tilde{x} \rangle - \| \varphi'(x^*) - \varphi'(\tilde{x} ) \| \|x - \tilde{x} \| \qquad \text{By Cauchy-Schwarz} \nonumber\\
& \geq \underbrace{\langle  \varphi'(x^*),x - x^* \rangle}_{\geq 0 \ \ \text{by optimality}} + \langle  \varphi'(x^*),x^* -\tilde{x} \rangle- \| \varphi'(x^*) - \varphi'(\tilde{x} ) \| \|x - \tilde{x} \| \qquad \text{Add and subtract $x^*$} \nonumber\\
& \geq \underbrace{\varphi(x^*) - \varphi(\tilde{x})}_{-\delta^{f}_k } + \underbrace{\frac{\sigma_0}{p}\|x^* - \tilde{x}\|^p}_{\geq 0 } - \underbrace{\| \varphi'(x^*) - \varphi'(\tilde{x} ) \|}_{\ \leq  (L_0 + M_1)\|x^* - \tilde{x}\|} \underbrace{\|x - \tilde{x} \|}_{\leq D_Q} \qquad \text{Using~\eqref{eq:prop_subproblem}}\nonumber \\  
& \geq -\delta^{f}_k  - (L_0 + M_1)\|x^* - \tilde{x}\|D_Q 
\end{align}
Also from~\eqref{eq:prop_subproblem}, we have:
\begin{align*}
\delta^{f}_k  \geq  \varphi(\tilde{x}) - \varphi(x^*) \geq \frac{\sigma_0}{p}\|\tilde{x} - x^*\|^p \qquad \text{then} \qquad  \|\tilde{x} - x^*\| \leq \left(\frac{\delta^{f}_k  \cdot p}{\sigma_0} \right)^{1/p}
\end{align*}

Therefore, from~\eqref{eq:aux_5}, we obtain:
\begin{align}\label{eq:bound_gradient}
\langle  \varphi'(\tilde{x}),x - \tilde{x} \rangle &  \geq -\delta^{f}_k  - (L_0 + M_1)\|x^* - \tilde{x}\|D_Q \nonumber \\
& \geq -\delta^{f}_k  - (L_0 + M_1)D_Q\left(\frac{\delta^{f}_k  \cdot p}{\sigma_0} \right)^{1/p} \nonumber\\
& \geq -(\delta^{f}_k )^{1/p} - (L_0 + M_1)D_Q\left(\frac{\delta^{f}_k  \cdot p}{\sigma_0} \right)^{1/p} \qquad \text{assuming $\delta^{f}_k <1$} \nonumber\\
& = -(\delta^{f}_k )^{1/p}\left(1+ (L_0+M_1)\left(\frac{p}{\sigma_0} \right)^{1/p} D_Q \right)\nonumber\\
& = -\tilde{\delta}^{f}_k 
\end{align}
The result in~\eqref{eq:bound_gradient} allows us to control the error in the gradient from the error in the function value.

Recall that we need to prove
\begin{align*}
\min_{x\in Q} \phi_{k+1}(x) \geq  f(\hat{x}_{k+1})  - \varepsilon.
\end{align*}
And so far we have
\begin{align}\label{eq:intermediate}
\min_{x\in Q} \phi_{k+1}(x) & \geq f(\hat{x}_{k+1})  -(1-\alpha_k)\big( \varepsilon +\bar{\delta}^{\phi}_k  \big) + \min_{x\in Q} \{ \frac{\lambda_{k+1}}{\alpha^3_k} \frac{M_2}{6}\|x -z_k\|^3  + \langle \nabla f (\hat{x}_{k+1}) , x - \hat{x}_{k+1}\rangle   \}.
\end{align}
So we continue this proof by bounding the last term above
\begin{align*}
\langle \nabla f (\hat{x}_{k+1}) , x - \hat{x}_{k+1}\rangle.
\end{align*}
It follows from~\eqref{eq:bound_gradient}
\begin{align}\label{eq:bound1}
-\tilde{\delta}^{f}_k & \leq \langle  \varphi'(\hat{x}_{k+1}),x - \hat{x}_{k+1} \rangle \nonumber\\
& \leq  \langle \nabla f(z)+ \nabla^2 f(z)(\hat{x}_{k+1}-z),x-\hat{x}_{k+1}\rangle + \frac{N \|z - \hat{x}_{k+1}\|}{2} \langle \hat{x}_{k+1} -z, x - \hat{x}_{k+1} \rangle    \nonumber \\
& \leq  \langle \nabla f(z)+ \nabla^2 f(z)(\hat{x}_{k+1}-z),x-\hat{x}_{k+1}\rangle + \frac{N \|z - \hat{x}_{k+1}\|}{2} \big( \|z - \hat{x}_{k+1}\| \|x-z\| - \|z - \hat{x}_{k+1}\|^2\big)  
\end{align}
On the other hand, from Hessian Lipschitz continuity, it follows that
\begin{align*} 
&\langle \nabla f(z) + \nabla^2f(z)(\hat{x}_{k+1} - z), x - \hat{x}_{k+1}\rangle \nonumber \\
&\leq \| \nabla f(z) + \nabla^2f(z)(\hat{x}_{k+1} - z) - \nabla f(\hat{x}_{k+1})  \| \|x - \hat{x}_{k+1} \| + \langle \nabla f(\hat{x}_{k+1}), x -\hat{x}_{k+1}\rangle \nonumber \\
& \leq \frac{M_2}{2}\|z - \hat{x}_{k+1}\|^2\|x - \hat{x}_{k+1}\|+ \langle \nabla f(\hat{x}_{k+1}), x -\hat{x}_{k+1}\rangle \nonumber \\
& \leq \frac{M_2}{2}\|z - \hat{x}_{k+1}\|^2(\|x - z\| + \|z - \hat{x}_{k+1}\|)+ \langle \nabla f(\hat{x}_{k+1}), x -\hat{x}_{k+1}\rangle  ,
\end{align*}
which implies
\begin{align}\label{eq:bound2}
&\frac{M_2}{2}\|z - \hat{x}_{k+1}\|^2(\|x - z\| + \|z - \hat{x}_{k+1}\|)+ \langle \nabla f(\hat{x}_{k+1}), x -\hat{x}_{k+1}\rangle  - \\
& \qquad  \langle \nabla f(z) + \nabla^2f(z)(\hat{x}_{k+1} - z), x - \hat{x}_{k+1}\rangle \geq 0.
\end{align}

Adding \eqref{eq:bound1} and \eqref{eq:bound2} we obtain:
\begin{align}\label{eq:intermediate2}
\frac{N+M_2}{2}\|z - \hat{x}_{k+1}\|\|x - z\|+\frac{M-N}{2}\|z - \hat{x}_{k+1}\|^3 + \langle \nabla f(\hat{x}_{k+1}), x -\hat{x}_{k+1}\rangle \geq - \tilde{\delta}^{f}_k.
\end{align}

Lets recall~\eqref{eq:intermediate}, and replacing in~\eqref{eq:intermediate2} we have:
\begin{align*}
\min_{x\in Q} \phi_{k+1}(x) & \geq f(\hat{x}_{k+1})  -(1-\alpha_k)\big( \varepsilon +\bar\delta^\phi_k \big) + \min_{x\in Q} \{ \frac{\lambda_{k+1}}{\alpha^3_k} \frac{M_2}{6}\|x -z_k\|^3  + \langle \nabla f (\hat{x}_{k+1}) , x - \hat{x}_{k+1}\rangle   \}.
\end{align*}
Now, lets focus on the second term above,
\begin{align*}
& \min_{x\in Q} \{ \frac{\lambda_{k+1}}{\alpha^3_k} \frac{M_2}{6}\|x -z_k\|^3  + \langle \nabla f (\hat{x}_{k+1}) , x - \hat{x}_{k+1}\rangle   \}
\\ 
&\geq 
\min_{x\in Q} \{ \frac{\lambda_{k+1}}{\alpha^3_k} \frac{M_2}{6}\|x -z_k\|^3  -\tilde{\delta}^{f}_k - \frac{N+M_2}{2}\|z - \hat{x}_{k+1}\|\|x - z\|+\frac{N-M}{2}\|z - \hat{x}_{k+1}\|^3  \}
\\ 
&\geq 
-\tilde{\delta}^{f}_k + \underbrace{\min_{x\in Q} \{ \frac{\lambda_{k+1}}{\alpha^3_k} \frac{M_2}{6}\|x -z_k\|^3   - \frac{N+M_2}{2}\|z - \hat{x}_{k+1}\|\|x - z\|+\frac{N-M}{2}\|z - \hat{x}_{k+1}\|^3  \}}_{\geq 0 \ \  \text{when}\ \  N=5M}
\end{align*}
Moreover,the choice of $N=5M$ guarantees that
\begin{align}\label{eq:bound_beta}
\frac{1}{12} \geq \frac{\alpha_k^3}{\lambda_{k+1}}.
\end{align}

As a final step, we need to control the errors
\begin{align*}
-(1-\alpha_k)\big( \varepsilon +\bar\delta^\phi_k \big) - \tilde{\delta}^{f}_k \leq \varepsilon.
\end{align*}
Recall that
\begin{align*}
\tilde{\delta}^{f}_k & = \delta_f^{1/3}\left(1+ (L_0+M_1)D_Q\left(\frac{3}{\sigma_0} \right)^{1/3}  \right) \\
\bar\delta^\phi_k & = \delta_\phi^{1/3}\left(1+ L_0\lambda_kD_Q \left(\frac{3}{\sigma_0 \lambda_k}\right)^{1/3} \right)
\end{align*}
\newline
Therefore it is enough to select $\delta_f$ and $\varepsilon_\phi$ as follows
\begin{align*}
0\leq \delta^\phi_k \leq \min \left\lbrace  1, \left(\frac{\alpha_k \gamma}{1-\alpha_k}\right)^3\left(\frac{\varepsilon}{1+D_QL_0\big(3\lambda_k^{2} / \sigma_0 \big)^{1/3}}\right)^3\right\rbrace \\
0\leq \delta^f_k \leq \min \left\lbrace  1,\left((1-\gamma)\alpha_k +\frac{1}{2}\right)^{3} \left(\frac{\varepsilon}{1+D_QL_0\big(3 / \sigma_0 \big)^{1/3}}\right)^3\right\rbrace.
\end{align*}
With the two above choices of ac accuracy for the subproblem, we obtain the final desire result 
\begin{align*}
\min_{x\in Q} \phi_{k+1}(x) \geq  f(\hat{x}_{k+1})  - \varepsilon.
\end{align*}

\end{proof}

With Lemma~\ref{lemma:main} at hand, we are finally ready to state and prove our main result. But first, lets recall an auxiliary results form~\cite{baes2009estimate} that will allow us to bound the rate of convergence.

\begin{lemma}[Lemma 8.1 in~\cite{baes2009estimate}]\label{lemma:baes81}
    Consider a sequence $\{\alpha_k: k \geq 0\}$ in $(0,1)$, and define $\lambda_0 = 1$, $\lambda_{k+1} = (1-\alpha_k)\lambda_k$ for every $k\geq  0$. If there exists a constant $\beta >0$ and an integer $p >0 $ for which $\alpha_k^p/\lambda_{k+1} \geq \beta$ for every $k\geq  0$, then, for all $K \geq 1$, $
    \lambda_K \leq \big({p}/\big({p+N\big( \beta \big)^{1/p}\big)} \big)^p$.
\end{lemma}

Next, we state our main auxiliary result about the complexity of Algorithm~\ref{alg:baes_modified}, which we use for the proof of convergence rate of Algorithm~\ref{alg:main}.

\begin{theorem}\label{thm:main_result_aux}
    Let Assumptions~\ref{assum:smooth1}, \ref{assum:smooth2} and \ref{assum:bounded} hold. Let $\varepsilon >0$ and $\gamma \in (0,1)$. Moreover, set $\delta_\phi$ and $\delta_f$ according to Lemma~\ref{lemma:main} and,
    \begin{align*}
    K \geq  \left \lceil{ 12 \left(\frac{1}{\varepsilon} \right)^{1/3}\left( f(x_0) -f^* + \frac{M}{6}\|x_0-x^*\|^3 \right)^{1/3} }\right \rceil.
    \end{align*}
    Then, the output of Algorithm~\ref{alg:baes_modified} has the following property: $f(\hat x_K) - f^* \leq \varepsilon$.
\end{theorem}

Theorem~\ref{thm:main_result_aux} states that if one is allowed to solve the subproblems of Algorithm~\ref{alg:baes_modified} with the prescribed accuracy, then the oracle complexity of Algorithm~\ref{alg:baes_modified} is $O(1/\varepsilon^{1/3})$ to reach an solution that is $\varepsilon$ away to the optimal function value.
\begin{proof}
    Initially, from Lemma~\ref{lemma:main} we have that 
        \begin{align*}
    \min_{x\in Q} \phi_k(x) & \geq f(\hat{x}_k) - \varepsilon  \qquad \forall k \geq 1.
    \end{align*}
    Then, from Proposition~\ref{prop:estimate_sequence}, it follows that
    \begin{align*}
    f(\hat x_k) {-} f^* \leq \lambda_k(\phi_0(x^*) {-}f^*) + \varepsilon.
    \end{align*}
    which in combination with Lemma~\ref{lemma:baes81} and~\eqref{eq:bound_beta} provides:
    \begin{align*}
    f(\hat x_k) {-} f^* \leq \left(\frac{3}{3+k\big( 1/12\big)^{1/3}} \right)^3(\phi_0(x^*) {-}f^*) + \varepsilon.
    \end{align*}
    and the desired result follows.
\end{proof}

Now that we have Theorem~\ref{thm:main_result_aux} at hand, we can prove our main result.

\begin{proof}(Theorem~\ref{thm:main_result})
The proof follows the same arguments as proof of Theorem~\ref{thm:main_result_aux}.

Recall that in Algorithm~\ref{alg:main} we define:
\begin{align*}
    \phi_{k{+}1}(\mathbf{x}) {=} (1{-}\alpha_k)\phi_k(\mathbf{x}) + \alpha_k \big( F(\mathbf{x}_{k{+}1}) + \langle \nabla F(\mathbf{x}_{k{+}1}),\mathbf{x}{-} \mathbf{x}_{k{+}1} \rangle \big)
\end{align*}
    Then, from Lemma~\ref{lemma:main} it follows that
        \begin{align*}
    \min_{\mathbf{x} \in Q^m \bigcap  \mathcal{Q}_{\tilde\varepsilon}} \phi_k(\mathbf{x}) & \geq F(\hat{\mathbf{x}}_k) - \varepsilon  \qquad \forall k \geq 1.
    \end{align*}
    Then, from Proposition~\ref{prop:estimate_sequence}, it follows that
    \begin{align*}
    F(\hat{\mathbf{x}}_k) {-} F^* \leq \lambda_k(\phi_0(\mathbf{x}^*) {-}F^*) + \varepsilon.
    \end{align*}
    which in combination with Lemma~\ref{lemma:baes81} and~\eqref{eq:bound_beta} provides:
    \begin{align*}
    F(\hat{\mathbf{x}}_k) {-} F^* \leq \left(\frac{3}{3+k\big( 1/12\big)^{1/3}} \right)^3(\phi_0(\mathbf{x}^*) {-}F^*) + \varepsilon.
    \end{align*}
    and the desired result follows by finding $k\geq 1$ such that the first term above is less than $\varepsilon$.
\end{proof}

%% file: sec_proof_comm.tex
In this section, we study the communication complexity of Algorithm~\ref{alg:main}. Recall that at each iteration, we are required to solve two auxiliary subproblems (Lines 10 and 13) in a distributed manner via Algorithm~\ref{alg:approx_subproblem}.

Note that we have defined our consensus constraints as $Ax=0$, with $A$ being the graph Laplacian obtained from the network. 

Lets consider the sets $\mathcal{A} {=} \{x \mid Ax{=}0\}$, and  $\mathcal{B} {=} \{x \mid \|Ax\|{=}0\}$. Then $\mathcal{A} {=} \mathcal{B}$. Thus, we compare the following two problems
\begin{align}
\min_{ \|Ax\| \leq 0} f(x) \label{problems1}\\
\min_{ \|Ax\| \leq \varepsilon} f(x) \label{problems2} 
\end{align}

\begin{proposition}\label{prop:constrains}
    Denote as $f_0^*$ as the optimal value of the optimization problem \eqref{problems1}, and $f_\varepsilon^*$ as the optimal value of the optimization problem \eqref{problems2}. Moreover assume $f_0^*$ and $f_\varepsilon^*$ are finite, and define as $y_0^*$ and  $y_\varepsilon^*$ the corresponding dual optimal solutions and that there is no duality gap. Then
    \begin{align}\label{eq:bounds_problems}
    y^*_\varepsilon \varepsilon \leq f^*_0 {-} f^*_\varepsilon \leq y^*_0\varepsilon.
    \end{align}
\end{proposition}
\begin{proof}
The proof of~\eqref{eq:bounds_problems} follows from the fact that:
\begin{align*}
f^*_0 & {=} \inf_x \{ f(x) {+} y^*_0 \|Ax\| \}\\
f^*_\varepsilon & {=} \inf_x \{ f(x) {+} y^*_\varepsilon (\|Ax\| {-}  \varepsilon ) \}
\end{align*}
moreover, denote $
q^*_0 (y )  {=}  \inf_x \{ f(x) {+} y\|Ax\| \}.
$
Then
\begin{align*}
f^*_0 {{-}} f^*_\varepsilon & {{=}} \inf_x \{ f(x) {{+}} y^*_0 \|Ax\| \} {{-}} \inf_x \{ f(x) {{+}} y^*_\varepsilon (\|Ax\| {{-}} \varepsilon )  \} \\
& {{=}} \inf_x \{ f(x) {{+}} y^*_0 \|Ax\| \} {{-}}\inf_x \{ f(x) {{+}} y^*_\varepsilon \|Ax \|  \}  {{+}} y^*_\varepsilon \varepsilon\\
& {=} q^*_0 (y^*_0 ) {-} q^*_0 (y^*_\varepsilon )   {+} y^*_\varepsilon \varepsilon \\
& \geq y^*_\varepsilon \varepsilon,
\end{align*}
because by definition $y^*_0 $ maximizes $q^*_0$. The other direction follows similarly. We can conclude that if we have a point $\hat{x}$ such that for $\delta>0$ it holds that
\begin{align*}
f(\hat{x}) {-} f^*_0 \leq \delta.
\end{align*}
Then, it follows from~\eqref{eq:bounds_problems} that
\begin{align}\label{eq:cost_of_approx}
f(\hat{x}) {-} f^*_\varepsilon {+}  f^*_\varepsilon{-} f^*_0 \leq \delta \nonumber\\
f(\hat{x}) {-} f^*_\varepsilon \leq \delta {+} f^*_0 {-} f^*_\varepsilon \leq \delta {+}  y_0^*\varepsilon,
\end{align}
where recall that $\varepsilon$ is the upper bound in the approximate consensus constraint, and $\delta$ is the optimality gap by which the exact consensus problem has been solved.
\end{proof}

Proposition~\ref{prop:constrains} shows that if we obtain a point $\hat{x}$ that is $\delta$ away from optimality in terms of function value with respect to the constraint $Ax=0$. Then, it will at most  $2\delta$ away from optimality in terms of function value with respect to the constraint $\|Ax\|\leq \delta /R$. This result will be important in analyzing the communication complexity of the proposed algorithm as we analyze our algorithm with respect to the set $\mathcal{Q}_{\tilde{\varepsilon}} {=} \{ \mathbf{x} \in \mathbb{R}^{nm} \mid \|\sqrt{\mathbf{W}} \hat{\mathbf{x}}\|_2 \leq \tilde{\varepsilon}\}$.

\begin{proof}(Theorem~\ref{thm:complexity_inner})
The main idea of this proof is to exploit the fact that the subproblem of minimizing the cubic regularized approximation in~\eqref{eq:aux_subproblem} is dual{-}friendly, as shown in Section~\ref{sec:cubics}. Algorithm~\ref{alg:approx_subproblem} is an adaptation to subproblem~\eqref{eq:generic_subproblem} in~\cite[Algorithm $5$]{Uribe2018a}, whose communication complexity is explicitly available in~\cite[Theorem $5.3$]{Uribe2018a}. However, there are some technical aspects we have to take care of first.
    
    Initially, \cite[Theorem $5.3$]{Uribe2018a} guarantees at the end of required number of iterations $T$ we obtain an approximate solution $\mathbf{h}^*(\tilde{\mathbf{w}}_T)$ such that
    \begin{align*}
    \boldsymbol{\Phi}(\mathbf{h}^*(\tilde{\mathbf{w}}_T)) {-}\boldsymbol{\Phi}_0^* \leq \delta \quad \text{and} \quad \|A\mathbf{h}^*(\tilde{\mathbf{w}}_T)\| \leq \delta/R.
    \end{align*} 
    But it is important to note that $\boldsymbol{\Phi}^*_0$ is the optimal value for the function~$\eqref{eq:generic_subproblem}$ with the linear constraint $ A\mathbf{h} {=} \boldsymbol{0}_{nm}$. Whereas we need an approximate solution with respect to $ \|A\mathbf{h}\| \leq \tilde{\varepsilon}$. It follows from Proposition~\ref{prop:constrains} that the point $\mathbf{h}^*(\tilde{\mathbf{w}}_T)$ has the following property:
    \begin{align*}
    \boldsymbol{\Phi}(\mathbf{h}^*(\tilde{\mathbf{w}}_T)) {{-}}\boldsymbol{\Phi}_{\tilde \varepsilon}^* \leq \delta {{+}} \tilde{\varepsilon}\|\mathbf{y}^*_0\| \ \text{and} \ \|A\mathbf{h}^*(\tilde{\mathbf{w}}_T)\| \leq \delta/R,
    \end{align*}
    where $\mathbf{y}^*_0$ is the optimal value of the dual function~\eqref{eq:dual_regu}. The desired result follows by setting $\tilde{\varepsilon} = \delta/R$.
\end{proof}

%% file: sec_experiments_arxiv.tex
In this section, we present numerical experiments for the implementation of Algorithm~\ref{alg:main} applied to the logistic regression problem: \begin{align}\label{eq:logistic loss}
        \min_{x \in \mathbb{R}^n}\frac{1}{d}\sum\limits_{i = 1}^d \ln\Bigl(1 + \exp\bigl(-y_i\langle w_i, x \rangle\bigr)\Bigr). 
\end{align} 
We given a set of $d$ data pairs $\{y_i,w_i\}$ for $1\leq i\leq d$, where $y_i \in \{1, -1\}$ is the class label of object $i$, and $w_i \in \mathbb{R}^n$ is the set of features of object $i$. Moreover, we assume the data points are uniformly split among $m$ nodes, connected over a network, such that each node has $d^i$ data points, and $d= m\cdot d^i$.



\begin{figure}[th!]
    \centering
    \includegraphics[trim=35 0 40 0, clip,width=1\linewidth]{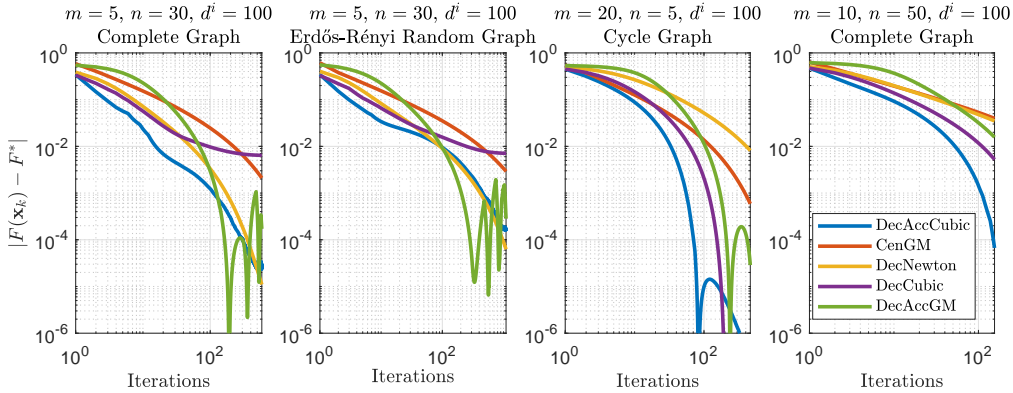}
    \vspace{-0.6cm}
    \caption{\textbf{Oracle Complexity}: Optimality gap of the iterations generated by Algorithm~\ref{alg:main} for different networks topologies and different problem parameters}
    \label{fig:oracle_complex}
\end{figure}

\begin{figure}[ht!]
    \centering
    \includegraphics[width=0.7\linewidth]{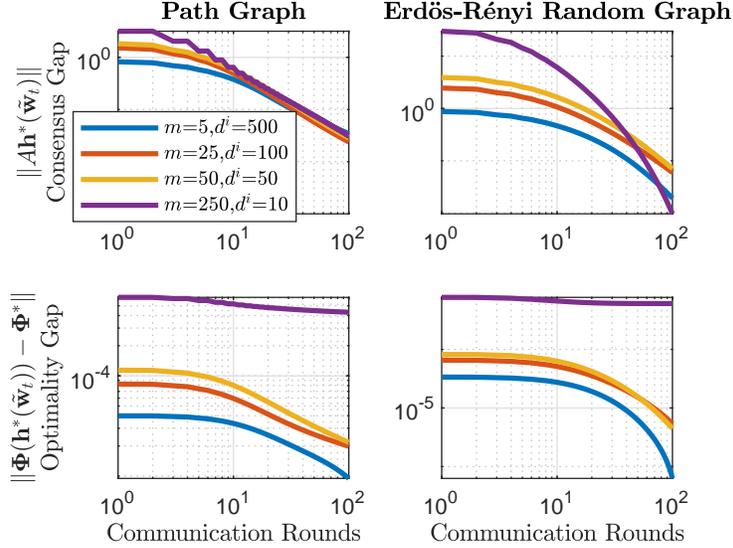}
    \caption{\textbf{Communication Complexity}: Consensus and optimality gap of the iterates generated by Algorithm~\ref{alg:approx_subproblem} for the solution of the auxiliary subproblems.}
    \label{fig:comm_complex}
\end{figure}

We compare the performance of Algorithm~\ref{alg:main} (\textbf{DecAccCubic}), centralized gradient method  (\textbf{CenGM}), distributed Newton method (\textbf{DecNewton}), distributed non-accelerated cubic method (\textbf{DecCubic}), and distributed accelerated gradient method (\textbf{DecAccGM}). Next, we describe each of these methods. 

Lets recall the two main optimization problems:
\begin{align*}
\min_{x\in Q } \left\lbrace f(x) \triangleq \sum_{i{=}1}^m f^i(x) \right\rbrace  \ \ \text{and} \ \
\min_{\substack{\mathbf{x} \in Q^m \\   \sqrt{\mathbf{W}}\mathbf{x}   {=} \boldsymbol{0}_{nm} } } \left\lbrace F(\mathbf{x}) \triangleq \sum_{i{=}1}^m f^i(x^i) \right\rbrace , 
\end{align*}

\textbf{Centralized gradient method (CenGM):} Gradient descent when all the data points are stored at the same location, i.e.,
\begin{align*}
    x_{k+1} = x_{k} - \alpha \sum_{i{=}1}^m \nabla f^i(x_k).
\end{align*}

\textbf{Distributed Newton method (DecNewton):} Netwon method with consensus contraints in the subproblem, i.e.,
\begin{align*}
      \mathbf{x}_{k+1} = \argmin_{A\mathbf{x} = 0}   \left\lbrace  F(\mathbf{x}_k) {+} \langle \nabla F(\mathbf{x_k}), \mathbf{x}{-}\mathbf{x_k} \rangle + \frac{1}{2} \langle \nabla^2 F(\mathbf{x_k})(\mathbf{x} {-}\mathbf{x_k}),\mathbf{x} {-}\mathbf{x_k}\rangle \right\rbrace
\end{align*}

 \textbf{Distributed non-accelerated cubic  (DecCubic):} Constrained Cubic Regularized Newton method with no acceleration
 
 \begin{align*}
      \mathbf{x}_{k+1} = \argmin_{A\mathbf{x} = 0}   \left\lbrace  F(\mathbf{x}_k) {+} \langle \nabla F(\mathbf{x_k}), \mathbf{x}{-}\mathbf{x_k} \rangle + \frac{1}{2} \langle \nabla^2 F(\mathbf{x_k})(\mathbf{x} {-}\mathbf{x_k}),\mathbf{x} {-}\mathbf{x_k}\rangle + \frac{N}{5}\|\mathbf{x}{-}\mathbf{x_k}\|^3\right\rbrace
\end{align*}

\textbf{Distributed accelerated gradient method (DecAccGM):} Accelerated Gradient Method with consensus constraints
\begin{algorithm}[H]
    \caption{Decentralized Accelerated Gradient Method}
    \label{alg:main2}
    \begin{algorithmic}[1]
        \STATE {\bfseries Input:} $x_0^i{=} \boldsymbol{0}_n$ $\upsilon_0^i{=} x_0^i$, $\lambda_0{=} 1$, $\forall i\in V$, $\phi_0(\mathbf{x}) {=} F(\mathbf{x}_0) {+} M_1\|\mathbf{x}{-}\mathbf{x}_0\|^2/2$, Number of iterations $K$.
        \STATE \textit{Each agent executes the following:}
        \FOR{$k{=}1,\cdots K-1$ }        
        \STATE Find $\alpha_k$ such that $\alpha_k^2 {=} (1{-}\alpha_k)\lambda_k$.
        \STATE $\lambda_{k{+}1} {=} (1{-}\alpha_k)\lambda_k$.
        \STATE $z^i_{k} {=} \alpha_k \upsilon_k^i {+} (1{-}\alpha_k)x_k^i$.
        \STATE \textit{Jointly solve:}
        \STATE $\mathbf{x}_{k+1} \approx_{\delta^F_k} \argmin_{A\mathbf{x} = 0}   \left\lbrace  F(\mathbf{x}_k) {+} \langle \nabla F(\mathbf{x_k}), \mathbf{x}{-}\mathbf{x_k} \rangle  + \frac{M_1}{2}\|\mathbf{x}{-}\mathbf{x_k}\|^2\right\rbrace$.
        \STATE $\phi_{k{+}1}(\mathbf{x}) {=} (1{-}\alpha_k)\phi_k(\mathbf{x}) + \alpha_k \big( F(\mathbf{x}_{k{+}1}) + \langle \nabla F(\mathbf{x}_{k{+}1}),\mathbf{x}{-} \mathbf{x}_{k{+}1} \rangle \big)$.
        \STATE \textit{Jointly solve:}
        \STATE $\boldsymbol{\upsilon}_{k{+}1}     \approx_{\delta^\phi_k} \argmin\limits_{\mathbf{x} \in Q^m \bigcap  \mathcal{Q}_{\tilde\varepsilon}} \phi_{k{+}1}(\mathbf{x})$.
        \ENDFOR
        \STATE {\bfseries Output:} End points $x_K^i, \ \forall i\in V$.
    \end{algorithmic}
\end{algorithm}
 Figure~\ref{fig:oracle_complex} shows the oracle complexity of Algorithm~\ref{alg:main} in terms of the optimality gap of the generated iterations for different network topologies (Complete, Erd\"os-R\'enyi, and Cycle graphs) and various problem parameters (number of agents, number of data points and dimensions). In all scenarios we explore, the proposed approach has the best performance with respect to the oracle complexity. Figure~\ref{fig:comm_complex} shows the communication complexity of solving the auxiliary subproblems with Algorithm~\ref{alg:approx_subproblem}. The top row shows the consensus gap, which indicates that the agreement among the agents on a solution increases as the number of communication rounds increases. The bottom row shows the agreement is on a solution to the auxiliary subproblem.

%% file: sec_discussion.tex
\textbf{Related Work:} Cubic-regularized Newton's method in the centralized setup has been extensively studied for large problem classes of convex and non{-}convex problems, e.g., Riemannian Manifolds in~\cite{zhang2018cubic,agarwal2018adaptive}, where it was shown that the proposed algorithm reaches a second{-}order $\varepsilon${-}stationary point within $O( \varepsilon^{-3/2})$ under certain smoothness conditions, which is optimal for the function classes. A stochastic variance{-}reduced cubic regularized newton methods was proposed in~\cite{zhou2018stochastic}, where it was shown that the proposed algorithm converges to an $(\varepsilon,\sqrt{\varepsilon}$){-}approximate local minimum within $\tilde{O}(n^{4/5}\varepsilon^{-3/2})$. In~\cite{wang2018stochastic} an iteration complexity of $O(n^{2/3}\varepsilon^{-3/2})$, and randomized blocks in~\cite{doikov2018randomized}. See~\cite{Jiang2017,cartis2011adaptive,cartis2011adaptive2,cartis2020concise} for a extensive treatment of Cubic regularization. Additionally, inexactness in cubic regularization has also been explored, ~\cite{ghadimi2017second,wang2018note} explores inexact Hessian,~\cite{song2019inexact} assumes geometric convergence rates in the approximate solution of the subproblem or inexact subproblem computation in the unconstrained case~\cite{Cartis2012}. In \cite{cartis2011adaptive,cartis2011adaptive2}, the authors explored adaptive methods to handle unknown Lipschitz constants and inexact problem solution, but only a $O(k^{-2})$ convergence rate was shown. In~\cite{Grapiglia2019,Grapiglia2019a} inexact solutions of high-order unconstrained problems and H\"older continuity were also explored.

\textbf{Inexact gradient and Hessian:} We assumed that each agent could compute the Hessian of the function stored locally in its memory, i.e., $\nabla^2f^i(x)$. In the case where $f^i(x) =\sum_{l=1}^d \ell_l(x)$ where $\ell_l$ is the loss function at a point $l$, this reduces the computation of the full Hessian $\nabla^2f(x)$, as each node can compute its local Hessian in parallel. However, as studied in~\cite{ghadimi2017second,wang2018note,Jiang2017}, even in such a reduced setting, the computational of the local Hessian might be practically or computationally intractable. Thus, effective ways to incorporate such inexactness should be studied. Studying the effects of an inexact gradient, Hessian, or function evaluation in distributed cubic regularized methods remains an open problem.

\textbf{Distributed implementation beyond strong convexity:} A main technical result in obtaining a distributed algorithm was described in Section~\ref{subsec:dec}. In short, the structure of the problem allows for distributed computation of the gradient of the dual function using local information only. However, only a sublinear convergence rate was achieved in the solution of the subproblem in Theorem~\ref{thm:complexity_inner}. Note that if the regularization term was quadratic, then linear rates could be achieved. The study of the primal-dual relationship between uniform convexity and H\"older continuity requires further study \cite[Lemma 1]{nesterov2015universal}, \cite[Lemma 1]{dvurechensky2017gradient}, or \cite[Corollary 18.14]{bauschke2011convex}.

\textbf{Reaching optimality in the distributed setting:} The convergence rate obtained by Algorithm~\ref{alg:main} is not optimal for the class of functions with Lipschitz Hessian. From first-order methods, it is known that optimal bounds are proportional to centralized lower bounds times a measure of connectivity of the network~\cite{Uribe2018a}, usually $O(m)$. However, optimal rates in high-order methods strongly depend on online search procedures~\cite{monteiro2013accelerated,pmlr-v99-gasnikov19a}. Such line search methods are not currently available for distributed methods.

\textbf{Distributed high-order methods:} Recently, implementable high-order methods have been proposed~\cite{kamzolov2020near,nesterov2020superfast}, where third-order information is approximated by second and first-order generating methods with very fast convergence rates, e.g. $O(k^{-5})$. The study of decentralization and inexactness for such methods require further study.

%% file: sec_conclusion.tex
In this paper, we developed a  second-order Newton-type method based on cubic regularization to minimize convex, finite-sum minimization problems over networks. With the additional assumption that the objective function has a Lipschitz Hessian, the convergence rate is shown to be ${O}(k^{-3})$, which improves on first-order distributed methods ${O}(k^{-2})$. The proposed algorithm extends the inexact cubic regularized Newton method~\cite{baes2009estimate} to the distributed setup, and shows that the auxiliary subproblems can be solved cooperatively and in a distributed manner over an arbitrary network by exploiting the primal-dual structure of the cubic terms.  Compared to centralized approaches, the achieved convergence rate is slightly sub-optimal, as lower bounds for second-order methods are known to be ${O}(k^{-7/2})$. However, the proposed algorithm is suitable for applications with distributed storage and computation capabilities spread over arbitrary networks. It is an open question of whether optimal rates can be achieved in a distributed setup. Moreover, further analysis of the proposed method's communication complexity is required, as we have focused on improving the oracle complexity (computations of gradients and Hessians) while guaranteeing a distributed, nearest-neighbor based implementation.